\documentclass[preprint,11pt]{elsarticle}
\usepackage[T1]{fontenc}
\usepackage[latin9]{inputenc}
\usepackage{multirow}
\usepackage{array} 
\usepackage{amsmath, amssymb}
\usepackage{empheq}

\usepackage{amsthm} 

 \newtheorem{theorem}{Theorem}[section]

\newtheorem{proposition}[theorem]{Proposition}
\newtheorem{remark}[theorem]{Remark}
\newtheorem{definition}[theorem]{Definition}

\usepackage{color}
\usepackage{graphicx,color}
\usepackage{amsmath, amssymb, graphics}

\def \r{\mathbb R}	
 \def \h{\mathbb H}

\def \s{\mathbb S}
\def \m{\mathbb M}

\begin{document}
\begin{frontmatter}
\title{ Classification of rotational surfaces with constant skew curvature in 3-space forms}
  \author[label1]{Rafael L\'opez}
 \ead{rcamino@ugr.es}
 \address[label1]{ Departamento de Geometr\'{\i}a y Topolog\'{\i}a\\ Instituto de Matem\'aticas (IEMath-GR)\\
 Universidad de Granada\\
 18071 Granada, Spain}
 \author[label2]{\'Alvaro P\'ampano}
 \ead{alvaro.pampano@ehu.eus}
 \address[label2]{Departament of  Mathematics\\ Faculty of Science and Technology\\
University of the Basque Country\\
48940 Bilbao, Spain}

\begin{abstract}
In this paper, we classify the rotational surfaces with constant skew curvature in $3$-space forms. We also give a variational characterization of the profile curves of these surfaces as critical points of a curvature energy involving the exponential of the curvature of the curve. Finally, we provide a converse process to produce all rotational surfaces with constant skew curvature based on the evolution of a critical curve under the flow of its binormal vector field with prescribed velocity.
 
\end{abstract}
\end{frontmatter}

\section{Introduction}

The objective of our investigation is the study and classification of the rotational surfaces in $3$-space forms such that $H^2-K$ is constant. Here $H$ is the mean curvature  and $K$ is the Gaussian curvature of the surface. In particular, we characterize the profile curves of these surfaces as critical points for an energy involving the exponential function of the curvature of the curve.

The curvature of a surface $N^2$ in the Euclidean space $\r^3$ measures how the surface is curved on the space. While $H$ is an extrinsic notion that reflects how $N^2$ is immersed in $\r^3$, the Gaussian curvature $K$ is intrinsic and can be measured within the surface. The quantity $H^2-K$ is always non-negative and vanishes only at umbilical points. This function $H^2-K$ appears in different areas of physics and chemistry. Perhaps, the most known one is related with the theory of elasticity of membranes, the region between two neighboring domains to which we can assign elastic energies. From early works starting with  Germain and Poisson, the elasticity energy density is proportional to $H^2$. Here we refer to the pioneering work of Helfrich \cite{hel}, where the energy of the shape of elastic lipid bilayers, such as biomembranes, involves the integral of $H^2-K$ and its minimizers satisfy a Schr\"{o}ndiger equation whose potential is proportional to $H^2-K$. 

More recently, the function $H^2-K$ has appeared in quantum mechanics in the study of the dynamic of a massive particle with mass $m$ constrained to move on a surface $N^2$. In such a case, the curvature of the surface induces a potential  that is determined by the geometry of the surface. This confining potential $V$ appears as a form acting in the normal direction to $N^2$ and it is the responsible for the constraint: see \cite{costa,encinosa,jensen,sbr}. This potential is precisely
$$V=-\frac{\hbar^2}{2m}\left(H^2-K\right)$$
and is called the `distortion potential' in \cite{encinosa} or the `geometry-induced potential' in \cite{sbr}. The motion of the particle on the given surface $N^2$ is governed by the Shr\"{o}ndiger type equation
$$\mbox{i}\hbar \frac{\partial\Psi}{\partial t}= -\frac{\hbar^2}{2m}\Delta_N\Psi+V \Psi$$
where $\Delta_N$ is the Laplace-Beltrami operator on $N^2$. Another setting where the effect of the potential $H^2-K$ appears is in optics. For example, the propagation of a monochromatic light wave constrained to move on a two dimensional layer, modeled as a surface $N^2$, is described by the scalar Helmholtz equation $\left(\Delta_N+k^2\right)\Psi+\left(H^2-K\right)\Psi=0$, where $k$ is the wave number of light (\cite{longui,mar,pfjz,sch}).

All these examples suggest that surfaces with constant skew curvature have interest in   physics  and it is reasonable to consider this type of surfaces into non-flat ambient spaces. In this paper we consider this more general context. Let $\m^3(\rho)$ be a $3$-space form, that is, a $3$-dimensional simply connected complete Riemannian manifold with constant sectional curvature $\rho$. If $\rho=0$ we recover the Euclidean space $\r^3$. On the other hand, if $\rho>0$, $\m^3(\rho)$ represents the round sphere $\s^3(\rho)$, while if $\rho<0$, $\m^3(\rho)$ is the hyperbolic space $\h^3(\rho)$.

For a surface $N^2$ immersed in $\m^3(\rho)$, the mean curvature $H$ and the (intrinsic) Gaussian curvature $K$ are defined, respectively, as 
\begin{equation}
H=\frac{\kappa_1+\kappa_2}{2}, \quad\quad K=\kappa_1\kappa_2+\rho,\label{H}
\end{equation}
where $\kappa_1$ and $\kappa_2$ are the principal curvatures of $N^2$. We define the {\it skew curvature} $\mathcal{S}_k$ of the surface $N^2$ as
\begin{equation}
\mathcal{S}_k=2\sqrt{H^2-K+\rho}\,.\label{skew}
\end{equation}
Note that $\mathcal{S}_k$ is well defined since $H^2\geq K-\rho$ with equality only at umbilical points. In fact, using \eqref{H}, \eqref{skew} yields to
$$\mathcal{S}_k=2\sqrt{\frac{\left(\kappa_1+\kappa_2\right)^2}{4}-\kappa_1\kappa_2}=\sqrt{\left(\kappa_1-\kappa_2\right)^2}=|\kappa_1-\kappa_2|\geq 0.$$
Thus the skew curvature $\mathcal{S}_k$ quantifies how much a surface deviates from being totally umbilical. 

\begin{definition} A surface $N^2$ has constant skew curvature $c\geq 0$ if $\mathcal{S}_k=c$ on $N^2$, or equivalently, 
\begin{equation}
\kappa_1(p)=\kappa_2(p)+c\label{lw}
\end{equation}
for all $p\in N^2$. Here we assume that $\kappa_1\geq\kappa_2$ on $N^2$.
\end{definition}

We now show two particular examples of surfaces with constant skew curvature, which can be considered as `trivial examples'. Firstly, the case $c=0$ in \eqref{lw} implies that $N^2$ is a totally umbilical surface. Totally umbilical surfaces are planes and spheres in $\r^3$,  spheres  in $\s^3(\rho)$ and  hyperbolic planes, spheres, equidistant surfaces and  horospheres in $\h^3(\rho)$. Other family of constant skew curvature surfaces appears when the two principal curvatures are constant, that is, the family of isoparametric surfaces. Isoparametric surfaces in 3-space forms were classified by Cartan proving that they are either totally umbilical surfaces ($c=0$) or circular cylinders \cite{ca1,ca2}. From now on we will discard both types of surfaces.  

In    geometry, the surfaces with constant skew curvature have also received its attention. For example, Chen studied the skew curvature (what he called `difference curvature') and obtained lower bounds of $\int_{N^2} \mathcal{S}_k$ for closed surfaces in terms of the genus of $N^2$ (\cite{chen}). The skew curvature was also used by Milnor to define a family of differential forms in surfaces (\cite{mi1,mi2}). More recently, surfaces with constant skew curvature have been studied  as examples of  closed M\"{o}bius forms (\cite{li}). In \cite{toda}, the authors  proved that the class of constant skew curvature  surfaces   does not contain any Bonnet surface, that is,  none of them can be isometrically deformed preserving the mean curvature. Without assuming the constancy of the skew curvature, the equation of prescribing skew curvature in rotational surfaces was studied in \cite{silva,sbr}. From a different viewpoint, the equation \eqref{lw} shows that the surfaces with constant skew curvature belong to the class of Weingarten surfaces, that is, surfaces where there is a functional relation $W(\kappa_1,\kappa_2)=0$ between the principal curvatures. The equation \eqref{lw} is a simple relation where $W(\kappa_1,\kappa_2)$ is linear. The literature on Weingarten surfaces is large and here we refer to \cite{bg,ho,ks,lo,lopa,mo1,pa} without to be a complete list.  

The purpose of this work is twofold. First,  we establish  a variational characterization  of the class of rotational surfaces with constant skew curvature.  This may seem surprising because in general characterizations of variational type are related with some type of energy  defined on the surface, the area likely being the most famous. In our case,   we characterize the profile curves (also called the generating  curves) of these surfaces  as the critical points in a suitable space of curves of an energy associated to the curvature of the curve. The second purpose of this paper is to give a   classification of all possible shapes of surfaces with constant skew curvature in the three 3-space forms. In this classification, we describe the different types of shapes and we will show some pictures of these surfaces. 

The scheme of this paper is the following. Let $N^2\subset\m^3(\rho)$ be a rotational surface and let $\gamma$ be the profile curve   of $N^2$. We view  $\gamma$ as a  curve contained in a totally geodesic surface of $\m^3(\rho)$ which is identified with a $2$-space form $\m^2(\rho)$. We will prove in Section \ref{sec2} that if $N^2$ has constant skew curvature $c>0$, then $\gamma$ is a critical point of the curvature energy
\begin{equation}\label{en-in}
\mathbf{\Theta}_\mu(\gamma)=\int_\gamma e^{\,\mu\kappa}\,ds
\end{equation}
for $\mu=-1/c$, where the space of curves of the variation is formed by (planar) curves contained in $\m^2(\rho)$. The Euler-Lagrange equation associated to $\mathbf{\Theta}_\mu$ is
\begin{equation}\label{eq-an}
\frac{d^2}{ds^2}\left( e^{\,\mu\kappa}\right)+\left(\kappa^2-\frac{\kappa}{\mu}+\rho\right) e^{\,\mu\kappa}=0.
\end{equation} 

Moreover, in Theorem \ref{converse}, we give the converse process by producing rotational surfaces with constant skew curvature by means of a critical point $\gamma \subset\m^2(\rho)$ of the energy functional $\mathbf{\Theta}_\mu$ with $\mu\not=0$. Exactly, we associate to $\gamma$ the Killing vector field in $\m^3(\rho)$ which uniquely extends the vector field along $\gamma$ defined by $\mathcal{I}=\mu e^{\,\mu\kappa}\,B$ in the direction of the (constant) binormal $B$ of $\gamma$. Then the immersion determined by the flow of the extension of $\mathcal{I}$ defines a rotational surface $N^2$ in $\m^3(\rho)$ such that $\gamma$ is a profile curve and $N^2$ has constant skew curvature $\mathcal{S}_k=\lvert-1/\mu\rvert$.
 
In Section \ref{sec3}, we study some qualitative properties of the critical points of $\mathbf{\Theta}_\mu$ that will simplify the work for the classification of the subsequent section. For this, we will study the orbits of the phase plane of an autonomous system associated to the critical points. Finally, in  Section \ref{sec4}, we give a description of the critical curves of $\mathbf{\Theta}_\mu$ in $\m^2(\rho)$. This gives rise to the classification of spherical rotational constant skew curvature surfaces, which depends on whether the surface meets or does not meet the axis of rotation $\zeta$. All these surfaces are smooth but for the points where they meet $\zeta$. In these points the surface is of class $\mathcal{C}^1$. The classification is the following:

\begin{enumerate}
\item Case that the surface meets (necessarily orthogonally) the axis $\zeta$.
\begin{enumerate}
\item Ovaloids. Convex surfaces whose shapes are like an oblate spheroid. We refer to their profile curves as oval type.
\item Vesicle type surfaces. Embedded closed surfaces where the two poles of the profile curves $\gamma$ are very close. The profile curves, usually referred as to simple biconcave type curves, present two inflection points.
\item Pinched spheroids. Limit case of vesicle type surfaces when the two poles coincide. The profile curves are going to be called figure-eight type curves.
\item Immersed spheroids. Closed surfaces that appear when the two poles pass theirselves through the axis $\zeta$. In this case, the profile curves are called  non-simple biconcave type curves.
\end{enumerate}
\item Case that  the surface does not meet the axis of rotation $\zeta$.
\begin{enumerate}

\item Cylindrical anti-nodoid type surfaces. Non embedded surfaces asymptotic to a circular cylinder. The profile curves $\gamma$, called borderline type curves, have a single loop facing away from the axis.
\item Anti-nodoid type surfaces. Non embedded surfaces that are periodic in the direction of the axis $\zeta$ and whose profile curves $\gamma$ have loops facing away from the axis. The profile curves are called orbit-like type curves.
\end{enumerate}
\end{enumerate}

\begin{figure}[hbtp]
\begin{center}
\includegraphics[width=.32\textwidth,height=0.15\textwidth]{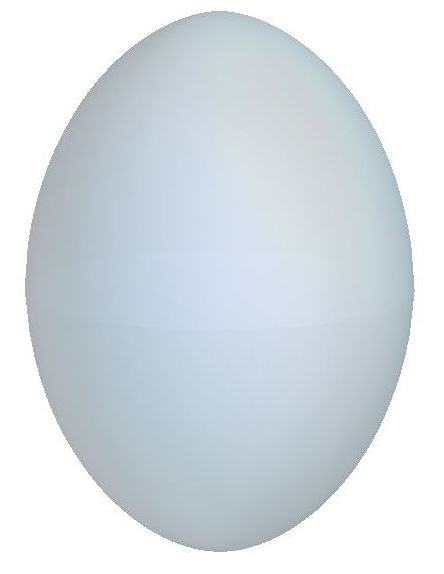}\,\,\,\includegraphics[width=.32\textwidth]{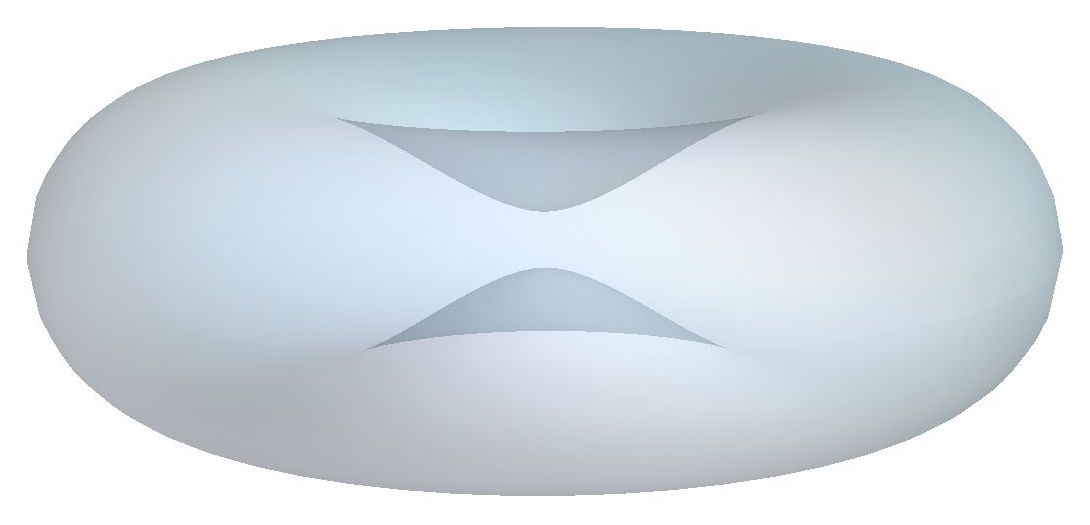}\,\,\,\includegraphics[width=.32\textwidth]{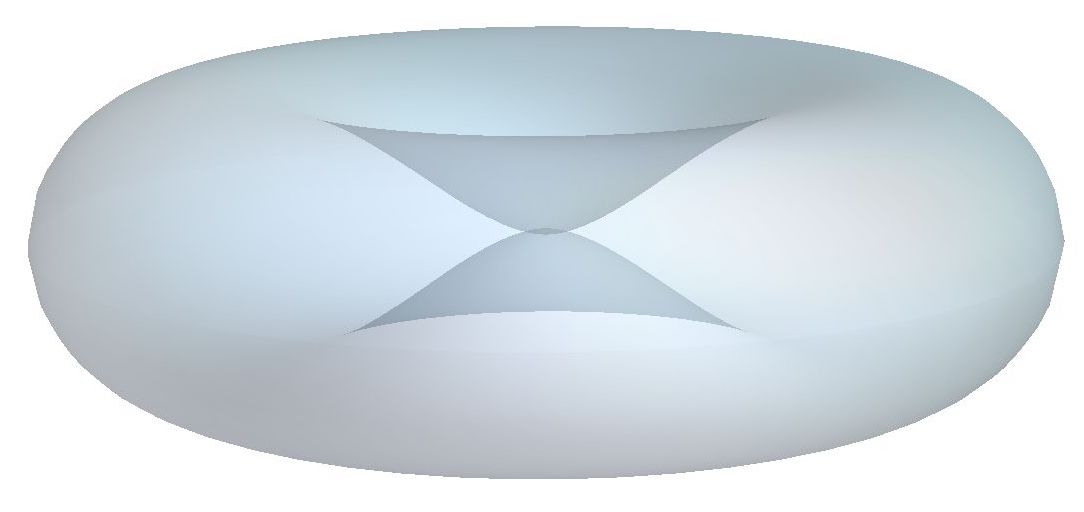}\\\vspace{0.5cm}\includegraphics[width=.32\textwidth,height=0.156\textwidth]{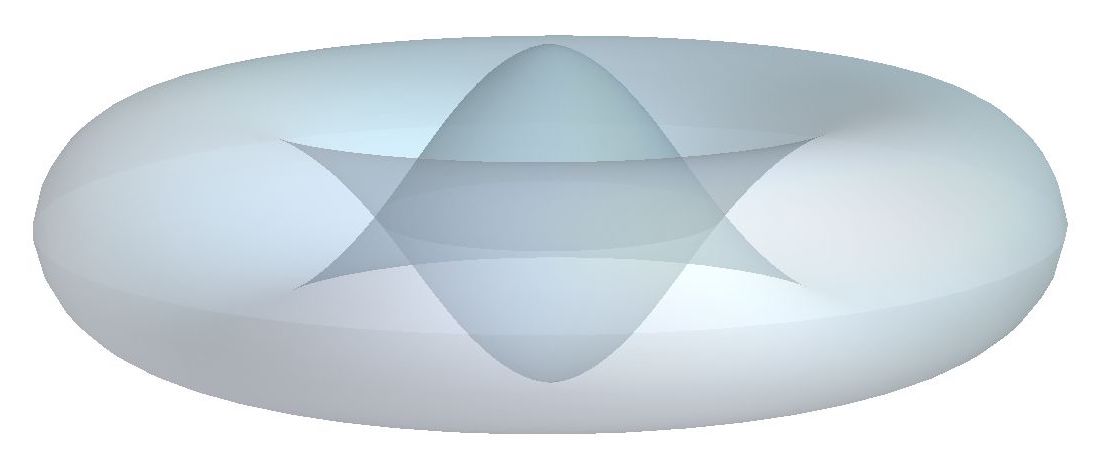}\,\,\,\includegraphics[width=.32\textwidth,height=0.16\textwidth]{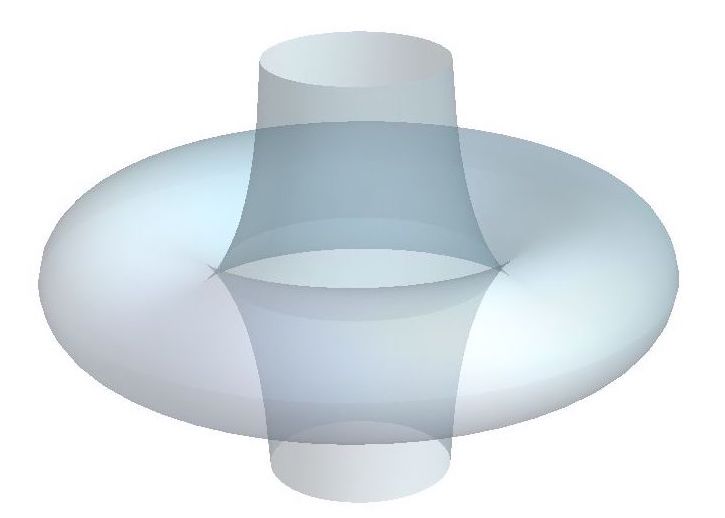}\,\,\,\includegraphics[width=.32\textwidth]{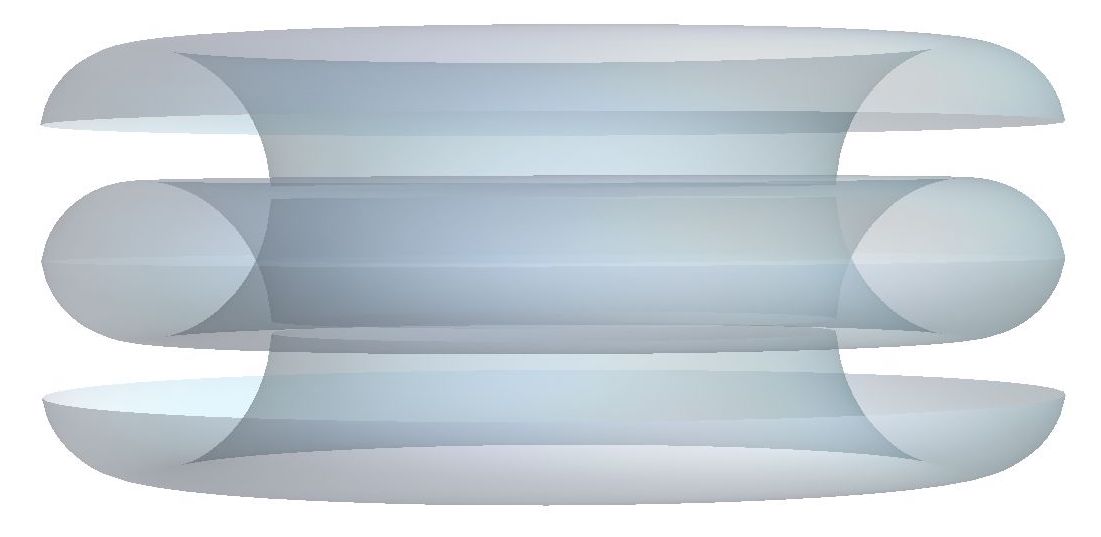}
\end{center}
\caption{Non-isoparametric rotational surfaces with constant skew curvature in the Euclidean space $\r^3$. From left to right: ovaloid, vesicle type, pinched spheroid, immersed spheroid, cylindrical anti-nodoid and anti-nodoid type surface.}
\label{surfacesR3}
\end{figure}

\section{Variational characterization of profile curves}\label{sec2}

As we have announced in the Introduction, the first goal of this paper is the variational characterization of the profile curves of rotational surfaces with constant skew curvature. In this section, we address this problem and in a first step, we prove in Theorem \ref{t1} that the profile curves satisfy the Euler-Lagrange equation of a curvature energy $\mathbf{\Theta}_\mu$, \eqref{en-in}. In a second step, we analyze in Proposition \ref{pr32} the critical points of this functional in a class of curves of $2$-dimensional space forms $\m^2(\rho)$ deriving a first integral of the Euler-Lagrange equation. Finally, we show in Theorem \ref{converse} a geometric construction for the converse process.

\subsection{Criticality of the profile curves}

Let $\m^3(\rho)$ be a $3$-space form and let $\langle,\rangle$ and $\widetilde{\nabla}$  denote the metric and the Levi-Civita connection on $\m^3(\rho)$, respectively.  Let  $\varphi:N^2\rightarrow \m^3(\rho)$ be an isometric immersion of  a surface $N^2$. We   identify $N^2$ by its image $\varphi(N^2)$ in $\m^3(\rho)$. If  $\nabla$ is the Levi-Civita connection of $\varphi$,   the   Gauss and Weingarten formulae are, respectively,
\begin{eqnarray}
\widetilde{\nabla}_X Y&=&\nabla_X Y+h\left(X,Y\right)\label{gauss}\\
\widetilde{\nabla}_X \eta&=&-A_\eta X+D_X^\perp \eta\,, \label{weingarten}
\end{eqnarray}
for any two tangent vector fields $X$ and $Y$ of  $N^2$. Here $h$ denotes the second fundamental form, $D^\perp$ is the connection on the normal bundle and $A_\eta$ stands for the  shape operator, where $\eta$ is a normal vector field to $N^2$. The  principal curvatures $\kappa_1$ and $\kappa_2$ are the eigenvalues of $A_\eta$. Then the mean curvature is defined by $H=\mbox{trace}(A_\eta)/2$ and, by the Gauss equation, the intrinsic Gaussian curvature is $K=\rho+\mbox{det}(A_\eta)$. This gives the expressions of $H$ and $K$ in \eqref{H}.
 
We now introduce the notion of rotationally symmetric surfaces in $\m^3(\rho)$. An immersed surface $N^2$ in $\m^3(\rho)$ is said to be a rotational surface if it is invariant under the action of a one-parameter group of rotations of $\m^3(\rho)$. This group of rotations leaves pointwise fixed a geodesic $\zeta$ called the rotation axis. When the group of all rotations with the same axis is isomorphic to $SO(2)\cong \mathbb{S}^1$, and consequently the orbits are circles, we say that $N^2$ is a spherical rotational surface. In the Euclidean space $\mathbb{R}^3$ and in the sphere $\mathbb{S}^3(\rho)$, all rotational surfaces are of this type. However,  in the hyperbolic space $\mathbb{H}^3(\rho)$, besides the spherical rotations, there are hyperbolic rotations and parabolic rotations. 

Let $N^2$ be a rotational surface invariant under the action of a one-parameter group of rotations $\{\phi_t\,:\,t\in\mathbb{R}\}$. A profile curve of $N^2$ is any curve $\gamma=\gamma(s)$, $s\in I\subset\r$, that is orthogonal to the orbits of the group of rotations. Then  $N^2$ can be parametrized as
\begin{equation}\label{ss}
N^2=\{\varphi(s,t)=\phi_{t} \left(\gamma(s)\right)\,:\,s\in I, t\in\mathbb{R}\}.
\end{equation}
Observe that $\gamma$ is fully contained in a totally geodesic surface of $\m^3(\rho)$, which is identified here with a $2$-space form $\m^2(\rho)$ with the same constant sectional curvature $\rho$.

We characterize the  profile curves of rotational surfaces with constant skew curvature as the critical points of a curvature energy functional.  

\begin{theorem}\label{t1} Let $N^2\subset \m^3(\rho)$ be a non-isoparametric rotational surface with constant skew curvature $\mathcal{S}_k=c>0$. If $\gamma$ is a profile curve of $N^2$, then the curvature $\kappa$ of $\gamma$ satisfies the Euler-Lagrange equation of the exponential type curvature energy
$$\mathbf{\Theta}_\mu(\gamma)=\int_\gamma e^{\,\mu\kappa},$$
where $\mu=-1/c$.
\end{theorem}

\begin{proof} Without loss of generality, we assume that $\gamma$ is parametrized by arc-length. For $N^2$, we use the parametrization $\varphi=\varphi(s,t)$ given in  \eqref{ss}. If we fix the variable $t$, the curves $\varphi(-,t)$ are geodesics of $N^2$ because they are congruent generated by evolving $\gamma$ under the group $\{\phi_t\}$. Moreover, these curves are orthogonal to the Killing vector field $\xi$ of the infinitesimal generator of $\{\phi_t\}$ and its curvature $\kappa(s,t)$ is $\kappa(s,t)=\kappa(s)$ for all $t$.

Let $G$ be the length of $\xi$, that is, $G^2(s)=\langle \phi_t,\phi_t\rangle$. Then, not only $G(s)$, but all the involved functions of $N^2$ depend only on the variable $s$ because $\kappa(s,t)=\kappa(s)$. In particular, the principal curvatures of $N^2$ are $\kappa_1(s)=-\kappa(s)$ and $\kappa_2(s)=h_{22}(s)$, respectively, where   
\begin{equation}\label{h22}
h_{22}=\frac{1}{\kappa}\left(\frac{G_{ss}}{G}+\rho\right)
\end{equation}
is the second coefficient of the second fundamental form $h$. See \cite{agp} for details. The subscript index $s$ indicates the derivative with respect to the arc-length parameter  $s$ of $\gamma$. Since $N^2$ is a non-isoparametric surface, it follows from \eqref{lw} that $\kappa_1(s)=-\kappa(s)\neq 0$, hence the curves $\varphi(-,t)$ are not geodesics of the ambient space $\m^3(\rho)$. Thus, we can   consider the Frenet frame of $\varphi(-,t)$, denoted by $\{T(s,t),N(s,t),B(s,t)\}$. 

Recall now that the Gauss-Codazzi equations are the compatibility conditions for the surface $N^2$ (\cite{agp}). In our setting, they reduce into a unique equation, namely 
\begin{equation}\label{gc}
\frac{\partial}{\partial s}\left(\frac{1}{\kappa}\left(G_{ss}+G\left(\kappa^2+\rho\right)\right)\right)-\kappa_s G=0\,.
\end{equation}
On the other hand, using the expressions of the principal curvatures in terms of the parametrization $\varphi$ and equation \eqref{h22}, the relation \eqref{lw} can be rewritten as
\begin{equation}\label{red}
G_{ss}+G\left(\kappa^2+c\kappa+\rho\right)=0\,.
\end{equation}
Since   $\kappa(s)$ is not constant, the Inverse Function Theorem implies that $s$ is a function of $\kappa$. Set $G(\kappa)=\dot{P}(\kappa)$, where the upper dot denotes the derivative with respect to $\kappa$. Now \eqref{gc} and \eqref{red} can be expressed, respectively, as  
\begin{eqnarray}
&&\dot{P}_{ss}+\dot{P}\left(\kappa^2+\rho\right)-\kappa\left(P+\lambda\right)=0 \label{arriba}\\
&&\dot{P}_{ss}+\dot{P}\left(\kappa^2+\rho\right)+c\kappa\dot{P}=0\,, \label{abajo}
\end{eqnarray}
for some real constant $\lambda\in\mathbb{R}$. Equation \eqref{arriba} turns out to be the Euler-Lagrange equation associated to a general curvature energy of the form
$$\mathbf{\Theta}(\gamma)=\int_\gamma\left(P(\kappa)+\lambda\right)ds,$$
where the constant $\lambda$ can be seen as a Lagrange multiplier constraining the length of the curve. By combining equations \eqref{arriba} and \eqref{abajo}, we deduce $c\dot{P}=-(P+\lambda)$. This equation can be integrated obtaining
$$P(\kappa)=e^{-\kappa/c}-\lambda\,,$$
up to a multiplicative constant which has no influence in the result. Hence, we draw the conclusion.
\end{proof}

\subsection{Exponential type curvature energy}

We now investigate the curvature energy $\mathbf{\Theta}_\mu$ that   appeared in  Theorem \ref{t1}. 

\begin{definition} Let $\mu$ be a real constant. For a curve $\gamma:I\rightarrow \m^2(\rho)$ in a $2$-space form $\m^2(\rho)$, we define the exponential type curvature energy 
\begin{equation}
\mathbf{\Theta}_\mu(\gamma)=\int_\gamma e^{\,\mu\kappa}.\label{energy}
\end{equation}
The constant $\mu$ is called the energy index.  
\end{definition}

Notice that if $\mu=0$, the energy $\mathbf{\Theta}_0$ is nothing but the length functional, whose critical points are, if properly parametrized, geodesics of $\m^2(\rho)$. Hence, from now on, we will assume that $\mu\not=0$. 

We compute the first variation formula of the   energy $\mathbf{\Theta}_\mu$ acting on the space of immersed curves in $\m^2(\rho)$. In the next computations, we will see these curves as planar curves in a $3$-space form $\m^3(\rho)$ where $\m^2(\rho)$ is embedded as a totally geodesic surface. 

Let $\gamma=\gamma(s)$ be a non-geodesic curve in $\m^2(\rho)$ which we suppose parametrized by the arc-length parameter $s$, with $s\in [0,L]$, being $L$ the length of $\gamma$. We see $\m^2(\rho)$ embedded in $\m^3(\rho)$.   If $T(s)=\gamma'(s)$ represents the unit tangent vector field, $\widetilde{\nabla}_T T(s)\not=0$ since $\gamma$ is not a geodesic and, hence, $\gamma$ has a well defined Frenet frame   $\{T,N,B\}$. The Frenet equations are
\begin{align*}
\widetilde{\nabla}_T T&=\kappa N\\
\widetilde{\nabla}_T N&=-\kappa T+\tau B\\
\widetilde{\nabla}_T B&=-\tau B\,,
\end{align*}
where $\tau$ is the torsion of $\gamma$. Since $\m^2(\rho)\subset\m^3(\rho)$, the rank of $\gamma$ is $2$ and $\tau=0$. In particular, the binormal $B=T\times N$ is constant on $\gamma$. 

For a sufficiently small $\varepsilon>0$, consider  a variation $\Gamma$ of $\gamma$, $\Gamma:[0,L]\times\left(-\varepsilon,\varepsilon\right)\rightarrow \m^2(\rho)$, $\Gamma=\Gamma(s,t)$ with $\Gamma(s,0)=\gamma(s)$. The variational vector field $W$ along $\gamma$ is $W(s)=\partial \Gamma/\partial t(s,0)$. Let $\Omega_{p_0p_1}$ be the space of regular curves in $\m^2(\rho)$ joining the endpoints $p_0=\gamma(0)$ and $p_1=\gamma(L)$. We consider the curves $\Gamma(-,t)$ parametrized by an arbitrary parameter $\sigma\in [0,1]$, which  are not necessarily parametrized by the arc-length. 

The computation of the first variation formula of $\mathbf{\Theta}_\mu$ involves the Frenet equations and general formulas for variations along the direction of $W$. Here we follow the seminal work of Langer and Singer in \cite{ls}: see also   \cite{agp} for a similar context. The first variation $\delta \mathbf{\Theta}_\mu(\gamma)$ of the curvature energy is 
\[\delta \mathbf{\Theta}_\mu(\gamma)=\frac{\partial}{\partial t}{\Big|}_{t=0}\mathbf{\Theta}_\mu\left(\Gamma(\sigma,t)\right)=\frac{\partial}{\partial t}{\Big|}_{t=0}\int_{0}^{1} e^{\,\mu\kappa} v \,d\sigma= \int_0^1 \left(\mu e^{\,\mu\kappa} v W(\kappa)+e^{\,\mu\kappa} W(v)\right) d\sigma\,,\]
where $v= v(\sigma,t)=\lvert \partial\Gamma/\partial\sigma (\sigma,t)\rvert$. We now differentiate with the aid of the variations of $v$ and $\kappa$ given in \cite[Lem. 1.1]{ls}, namely,
\begin{eqnarray*}
&&W(v)=v\langle \widetilde{\nabla}_TW,T\rangle\\
&&W(\kappa)=\langle \widetilde{\nabla}_T^2W,N\rangle-2\kappa\langle \widetilde{\nabla}_T W,T\rangle+\rho\langle W,N\rangle\,,
\end{eqnarray*}
and reparametrize by arc-length to obtain
\begin{eqnarray*}
\delta \mathbf{\Theta}_\mu(\gamma)&=&\int_0^L \Big(\mu e^{\,\mu\kappa} \left( \langle \widetilde{\nabla}_T^2W,N\rangle-2\kappa\langle \widetilde{\nabla}_TW,T\rangle+\rho\langle W,N\rangle\right)+e^{\,\mu\kappa} \langle \widetilde{\nabla}_TW,T\rangle \Big) ds\\
&=&\int_0^L e^{\,\mu\kappa} \Big(\mu\langle \widetilde{\nabla}_T^2W,N\rangle+\left(1-2\mu\kappa\right)\langle \widetilde{\nabla}_TW,T\rangle +\rho\mu\langle W, N\rangle\Big)ds\\
&=&\int_0^L \langle \widetilde{\nabla}_T^2\left(\mu e^{\,\mu\kappa}N\right)-\widetilde{\nabla}_T\left(\left(1-2\mu\kappa\right)e^{\,\mu\kappa}T\right)+\rho\mu N,W\rangle\,ds+\mathcal{B}\left(W,\gamma\right).
\end{eqnarray*}

Last equality holds after integrating by parts, where the boundary terms are included in the operator $\mathcal{B}\left(W,\gamma\right)$. If we define $\mathcal{E}(\gamma)$ as
$$\mathcal{E}(\gamma)=\frac{d^2}{ds^2}\left(e^{\,\mu\kappa}\right)+\left(\kappa^2-\frac{\kappa}{\mu}+\rho\right)e^{\,\mu\kappa},$$
then the expression $\delta\mathbf{\Theta}_\mu(\gamma)$ simplifies into
$$\delta \mathbf{\Theta}_\mu(\gamma)= \mu\int_0^L \mathcal{E}(\gamma)\langle W,N\rangle\,ds+\mathcal{B}\left(W,\gamma\right).$$
Regardless of the boundary terms, by standard arguments, a critical point for the energy $\mathbf{\Theta}_\mu$ must satisfy the equation $\mathcal{E}(\gamma)=0$, 
obtaining the announced Euler-Lagrange equation  \eqref{eq-an}.

\begin{proposition} \label{pr32}
Let $\gamma=\gamma(s)$ be a non-geodesic curve in $\m^2(\rho)$ parametrized by arc-length. If $\gamma$ is a critical point of $\mathbf{\Theta}_\mu$, then  
\begin{equation}
\frac{d^2}{ds^2}\left( e^{\,\mu\kappa}\right)+\left(\kappa^2-\frac{\kappa}{\mu}+\rho\right) e^{\,\mu\kappa}=0. \label{EL}
\end{equation}
This  equation  is called the Euler-Lagrange equation of $\mathbf{\Theta}_\mu$. Regardless of the boundary conditions, any curve satisfying \eqref{EL} is going to be called a critical curve of $\mathbf{\Theta}_\mu$.
\end{proposition}

A first property is that the constant $\mu$ can be chosen to be positive after a change of orientation of the critical curve $\gamma$, if necessary.
 
\begin{proposition}\label{orientation} If $\gamma\subset \m^2(\rho)$ is a critical curve for $\mathbf{\Theta}_\mu$, then the curve $\widetilde{\gamma}$ obtained by reversing the orientation of $\gamma$ is a critical curve for $\mathbf{\Theta}_{\widetilde{\mu}}$ with $\widetilde{\mu}=-\mu$.
\end{proposition}
\begin{proof} Let $\widetilde{\gamma}(s)=\gamma(-s)$. Then the curvature $\widetilde{\kappa}$ of $\widetilde{\gamma}$ is $\widetilde{\kappa}(s)=-\kappa(-s)$. By substituting in \eqref{EL}, we find
$$\frac{d^2}{ds^2}\left(e^{-\mu\widetilde{\kappa}}\right)+\left(\widetilde{\kappa}^2+\frac{\widetilde{\kappa}}{\mu}+\rho\right) e^{-\mu\widetilde{\kappa}}=0.$$
Thus  $\widetilde{\gamma}$ satisfies the Euler-Lagrange equation of $\mathbf{\Theta}_{\widetilde{\mu}}$ for the energy index $\widetilde{\mu}=-\mu$.      
\end{proof}

In what follows, we will assume that  the energy index $\mu$ is positive. Then, in the following result, we analyze the solutions of the Euler-Lagrange equation with constant curvature.

\begin{proposition}\label{constant} 
Let $\gamma$ be a curve in $\m^2(\rho)$ with constant curvature   satisfying the Euler-Lagrange equation \eqref{EL}. \begin{enumerate}
\item Case $\mathbb{R}^2$. Then $\gamma$ is either a geodesic curve   or a circle of radius $\mu$.
\item Case $\mathbb{S}^2(\rho)$. Then   $0<2\sqrt{\rho}\,\mu \leq1$. If $ 2\sqrt{\rho}\,\mu <1$, there are two solutions corresponding with two parallels. If $2\sqrt{\rho}\,\mu=1$, the solution is a circle with curvature $\sqrt{\rho}$.
\item Case $\mathbb{H}^2(\rho)$. Then $\gamma$ is either a circle or a hypercycle.
\end{enumerate}
\end{proposition}

\begin{proof}
If the curvature $\kappa$ of $\gamma$ is constant, we deduce from \eqref{EL} that $\kappa^2-\kappa/\mu+\rho=0$. If they exist, the two solutions $\kappa_0^+$ and $\kappa_0^{-}$ of this equation, which may coincide, are 
$$\kappa_0^{\pm}=\frac{1\pm \sqrt{1-4\rho\mu^2}}{2\mu}\,.$$
In particular, for the existence, $4\rho\mu^2\leq 1$ must hold and we have the following cases depending on the sign of $\rho$.
\begin{enumerate}
\item Case $\mathbb{R}^2$. Then $\kappa_0^{-}=0$ and $\gamma$ is a straight line and $\kappa_0^+=1/\mu$ and $\gamma$ is a circle of radius $\mu$.
\item Case $\mathbb{S}^2(\rho)$. We have the restriction that $1-4\rho\mu^2\geq 0$. Then $\kappa_0^+$ and $\kappa_0^-$ are both positive and  correspond with two circles (parallels after suitable rigid motion) if $2\sqrt{\rho}\,\mu<1$, or just one if equality holds. In the latter, $\kappa_0^+=\kappa_0^-=\sqrt{\rho}$.
\item Case $\mathbb{H}^2(\rho)$. Then $1-4\rho\mu^2> 0$, hence we always have two different solutions. For $\kappa_0^+$, we have  $\left(\kappa_0^+\right)^2>-\rho$ and this proves that the solution corresponds with an (Euclidean) circle. For $\kappa_0^{-}$, we have $\left(\kappa_0^-\right)^2<-\rho$, and the solution is a hypercycle.
\end{enumerate}
\end{proof}

We finish this subsection obtaining a  first integral of  the Euler-Lagrange equation \eqref{EL}.

\begin{proposition} \label{pr-first}
If $\gamma$ is a critical curve of $\mathbf{\Theta}_\mu$ with non constant curvature $\kappa$, then there is a constant $d\in\r$ such that 
\begin{equation}\label{fi}
\mu^4\kappa_s^2=d e^{-2\mu\kappa}-\left(\mu\kappa-1\right)^2-\rho\mu^2\,.
\end{equation}
\end{proposition}

\begin{proof}  Since the curvature $\kappa(s)$ is not constant, the derivative of $e^{\,\mu\kappa}$ with respect to $s$ does not vanish identically. We multiply equation \eqref{EL} by twice this derivative, obtaining 
\begin{eqnarray*}
0&=&2\left(e^{\,\mu\kappa}\right)_s\left(e^{\,\mu\kappa}\right)_{ss}+2\left(\kappa^2-\frac{\kappa}{\mu}\right) e^{\,\mu\kappa}\left(e^{\,\mu\kappa}\right)_s+2\rho e^{\,\mu\kappa}\left(e^{\,\mu\kappa}\right)_s\\ 
&=&2\left(e^{\,\mu\kappa}\right)_s\left(e^{\,\mu\kappa}\right)_{ss}+2\left(\kappa-\frac{1}{\mu}\right)e^{\,\mu\kappa}\left(\left(\kappa-\frac{1}{\mu}\right)e^{\,\mu\kappa}\right)_s+2\rho e^{\,\mu\kappa}\left(e^{\,\mu\kappa}\right)_s\\ 
&=&\left(\left(\left(e^{\,\mu\kappa}\right)_s\right)^2+\left(\kappa-\frac{1}{\mu}\right)^2 e^{2\mu\kappa}+\rho e^{2\mu\kappa}\right)_s=\frac{1}{\mu^2}\left(\left(\mu^4\kappa_s^2+\left(\mu\kappa-1\right)^2+\rho\mu^2\right)e^{2\mu\kappa}\right)_s.
\end{eqnarray*}
 Therefore,  there exists a constant $d$ such that 
$$\left(\mu^4\kappa_s^2+\left(\mu\kappa-1\right)^2+\rho\mu^2\right)e^{2\mu\kappa}=d,$$
concluding with \eqref{fi}.  
\end{proof}

\subsection{Geometric construction}

We finish this section showing a geometric method of constructing all rotational surfaces with constant skew curvature in 3-space forms. The idea is to use a critical curve $\gamma$ of the curvature functional $\mathbf{\Theta}_\mu$ and to evolve it in the direction of its binormal vector field with a prescribed velocity. Here we use the theory of binormal flow developed in \cite{agp} and which is based on a nice result of Langer and Singer  about Killing vector fields along $\gamma$ (\cite{ls}).  

Let $\gamma$ be a solution of the Euler-Lagrange equation \eqref{EL} parametrized by arc-length and consider $\m^2(\rho)$ isometrically immersed in $\m^3(\rho)$ as a totally geodesic surface. By Proposition \ref{orientation}, we may assume that $\mu>0$. Along $\gamma$,  define the vector field
\begin{equation}\label{I}
\mathcal{I}=\mu e^{\,\mu\kappa}\,B,
\end{equation}
  where $B$ denotes the (constant) binormal vector field of $\gamma$. Following \cite{ls}, $\mathcal{I}$ is a Killing vector field along $\gamma$ in the sense that $\gamma$ evolves in the direction of $\mathcal{I}$ without changing shape, only position. The Killing vector fields along curves can uniquely be extended to Killing vector fields on the whole ambient space, hence,  we extend $\mathcal{I}$ along $\gamma$ on $\m^3(\rho )$ and denote it by $\mathcal{I}$ again. 
Since $\m^3(\rho)$ is complete, let $\{\phi_t\,:\,t\in\mathbb{R}\}$ be the one-parameter group of isometries determined by the flow of $\mathcal{I}$. Define the \emph{binormal evolution surface} as 
\begin{equation}\label{nn}
N^2_\gamma=\{\varphi(s,t)=\phi_t(\gamma(s))\,:\, s\in I, t\in\mathbb{R}\}.
\end{equation}
Then $N^2_\gamma$ is an $\mathcal{I}$-invariant surface which is foliated by congruent copies of $\gamma$, $\gamma_t(s)=\phi_t(\gamma(s))$. Since $\phi_t$ are isometries of $\m^3(\rho)$, the curvature of any curve $\varphi(-,t)$ is   $\kappa(s)$. Similarly, all the curves   $\varphi(-,t)$ have zero torsion if $\gamma$ is not a geodesic in $\m^3(\rho)$. If    $\kappa$ is constant, then $N^2_\gamma$ is a flat isoparametric surface and if $\kappa$ is not constant, then $N^2_\gamma$ is a rotational surface of $\m^3(\rho)$: see  \cite{agp}.

We are in conditions to   prove the converse of Theorem \ref{t1}. 

\begin{theorem}\label{converse}  
 If $\gamma$ is  a critical curve for $\mathbf{\Theta}_\mu$   with non-constant curvature, then the binormal evolution surface $N^2_\gamma$ defined in \eqref{nn} is a rotational surface of $\m^3(\rho)$ with constant skew curvature $\mathcal{S}_k=1/ \mu$.
\end{theorem}

\begin{proof}
 Let $V(s)=\langle \mathcal{I},\mathcal{I}\rangle^{1/2}$ be the velocity of the flow, that is,
\begin{equation}\label{V}
V^2(s)=\langle \mathcal{I},\mathcal{I}\rangle=\langle \frac{\partial\varphi}{\partial t},\frac{\partial\varphi}{\partial t}\rangle=\mu^2 e^{2\mu\kappa}\,.
\end{equation}
Since the evolution is defined by isometries, $\gamma$ and all its congruent copies $\varphi(-,t)$ are planar critical curves for $\mathbf{\Theta}_\mu$ with non-constant curvature. From   \eqref{EL}, the function $V(s)$ satisfies 
$$V_{ss}+\left(\kappa^2-\frac{\kappa}{\mu}+\rho\right)V=0\,.$$
Since $N^2_\gamma$ is a rotational surface, the principal curvatures are $\kappa_1(s)=-\kappa(s)$ and $\kappa_2(s)=h_{22}(s)$, \cite{agp}. Thus $N^2_\gamma$ satisfies the equation \eqref{lw} for $c=-1/\mu$, hence $\mathcal{S}_k=\lvert c\rvert$ is constant, proving the result.
\end{proof}

Consequently, Theorem \ref{converse} provides a geometric method of constructing non-isoparametric rotational surfaces of constant skew curvature in $\m^3(\rho)$. Moreover, together with Theorem \ref{t1}, we have a complete characterization of these surfaces as binormal evolution surfaces generated by critical curves for the exponential type curvature energy $\mathbf{\Theta}_\mu$. Let us notice  that this characterization is valid for any of the three types of rotations in hyperbolic space because Theorem \ref{converse} holds  in these cases. Exactly,  the constant of integration $d$ in \eqref{fi}  plays a fundamental role to describe the shape of the orbits.  

\begin{proposition}
 Let $N^2_\gamma$ be a non-isoparametric rotational surface of $\m^3(\rho)$ with constant skew curvature and given by \eqref{nn}. If the constant of integration $d$ in \eqref{fi} is positive, then $N^2_\gamma$ is a spherical rotational surface.
\end{proposition}
\begin{proof}
We know that  equation \eqref{fi} represents a first integral of \eqref{EL}. For any fixed $s_0\in I$, let $\delta(\upsilon)$ represent an arc-length parametrized orbit of $N^2_\gamma$ through $s_0$, and denote by $\eta_\delta$ the curvature of $\delta$. Now, specializing the computations of \cite{agp} about   $\eta_\delta$, we have
$$\eta_\delta^2=\mu^4\kappa_s^2+h_{22}^2=\frac{d}{\mu^2 e^{2\mu\kappa}}-\rho\,,$$
where for the last equality we have used \eqref{h22} and \eqref{fi}.

Note that if $\rho\geq 0$,   the constant $d$ in \eqref{fi} is necessarily positive. In these cases, namely, $\r^2$ and $\s^2(\rho)$, there exist only spherical rotations. On the other hand, in $\mathbb{H}^2(\rho)$, the three possible signs of $d$ indicate the type of the three different options for the rotation. In particular, we have that orbits are circles ($\eta_\delta^2>-\rho$)  if and only if  $d>0$, as stated. 
\end{proof} 

From now on, the constant $d$ will be restricted to be positive.

\section{Properties of the critical curves}\label{sec3}

In this section, we obtain some qualitative properties of the critical curves $\gamma$ for  $\mathbf{\Theta}_\mu$ viewing $\gamma$ as solutions of an autonomous system. This will simplify the work for the subsequent classification of the rotational surfaces with constant skew curvature given in Section \ref{sec4}. 

We consider the Riemannian $2$-space form  $\m^2(\rho)$ as a subset of the affine space $\mathbb{R}^3$ given as follows. If   $(x_1,x_2,x_3)$ are the  canonical coordinates of $\r^3$, then the Euclidean plane $\mathbb{R}^2$ is $\mathbb{R}^2\times\{0\}$, the $2$-sphere $\mathbb{S}^2(\rho)$ is the hyperquadric $x_1^2+x_2^2+x_3^2=1/\rho$ and the hyperbolic plane $\mathbb{H}^2(\rho)$ is $x_1^2+x_2^2-x_3^2=1/\rho$ with $x_3>0$. For these models,   $\mathbb{R}^2$ and $\mathbb{S}^2(\rho)$ have the induced metric from the Euclidean metric of $\mathbb{R}^3$, while $\mathbb{H}^2(\rho)$ is endowed with the Lorentzian metric  $dx_1^2+dx_2^2-dx_3^2$.

Let $\gamma$ be a critical curve for $\mathbf{\Theta}_\mu$ with non-constant curvature $\kappa$. By Proposition \ref{pr-first}, we know that $\kappa$ satisfies  \eqref{fi} for some constant $d$. Recall that we are assuming $d>0$. Let us introduce the notation
\begin{equation}\label{not}
x=e^{\,\mu\kappa}\,,\quad \quad\quad y=x_s\,.
\end{equation}
Then the equation \eqref{fi} can be rewritten as an implicit equation that describes the orbits in the phase space,
\begin{equation}\label{F}
F(x,y)=\mu^2y^2+\left(\log x-1\right)^2x^2+\rho \mu^2 x^2=d\,,
\end{equation}
for $x>0$. 

Using \eqref{fi}, we analyze the orbits in the phase plane of the following autonomous system   
\begin{equation}\label{php}
\left(\begin{array}{c}x_s\\ y_s\end{array}\right)=\left(\begin{array}{c} y\\ -\dfrac{x}{\mu^2}\left(\log^2(x)-\log(x)+\rho\mu^2\right)\end{array}\right).
\end{equation}
The are two stationary points, namely, $P_+$ and $P_{-}$, where $P_{\pm}=\left(x_{\pm},0\right)$  and 
\begin{equation}\label{ll}
\log x_{\pm}=\frac{1\pm\sqrt{1-4\rho\mu^2}}{2}.
\end{equation}
See figure \ref{orbitas}. As it is expected, these points coincide with the solutions of the Euler-Lagrange equation \eqref{EL} with constant curvature. The Hessian matrix associated to \eqref{php} evaluated at $P_{\pm}$ is
$$\mathcal{H}\left(P_{\pm}\right)=\left(\begin{array}{cc} 0&1\\ \dfrac{\mp\sqrt{1-4\rho\mu^2}}{2\mu^2}&0\end{array}\right).$$
Then, the classification of the type of stationary points is done according to the value of $\rho$:
\begin{enumerate}
\item Case $4\rho\mu^2<1$. This occurs always in $\r^2$ and $\h^2(\rho)$.  
 The eigenvalues corresponding to the point $P_+$ are two opposite pure imaginary complex numbers and $P_{-}$ has two opposite real eigenvalues. Thus $P_{+}$ is a centre and $P_{-}$ is an unstable saddle point. 
\item Case $4\rho\mu^2= 1$. In this case, necessarily $\rho>0$. Now there is a unique stationary point, namely, $P_+=P_{-}=(\sqrt{e},0)$. The eigenvalues of $\mathcal{H}\left(P_+\right)$     are $0$ and the matrix $\mathcal{H}\left(P_+\right)$ is not diagonalizable. Then $(\sqrt{e},0)$ is a degenerate point.
\end{enumerate}

\begin{figure}[hbtp]
\begin{center}
\includegraphics[width=.3175\textwidth]{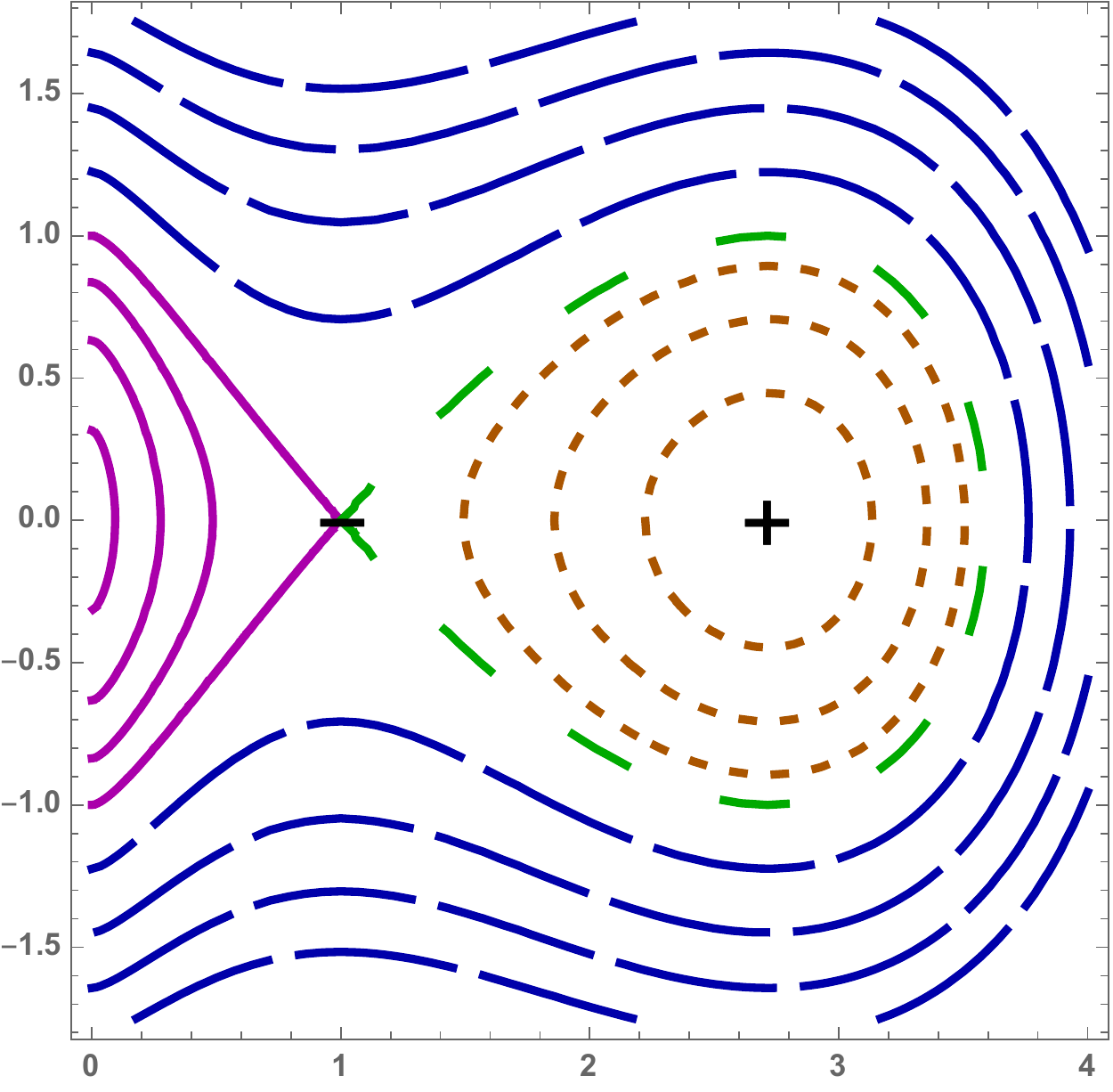}\quad\,\,\includegraphics[width=.31\textwidth]{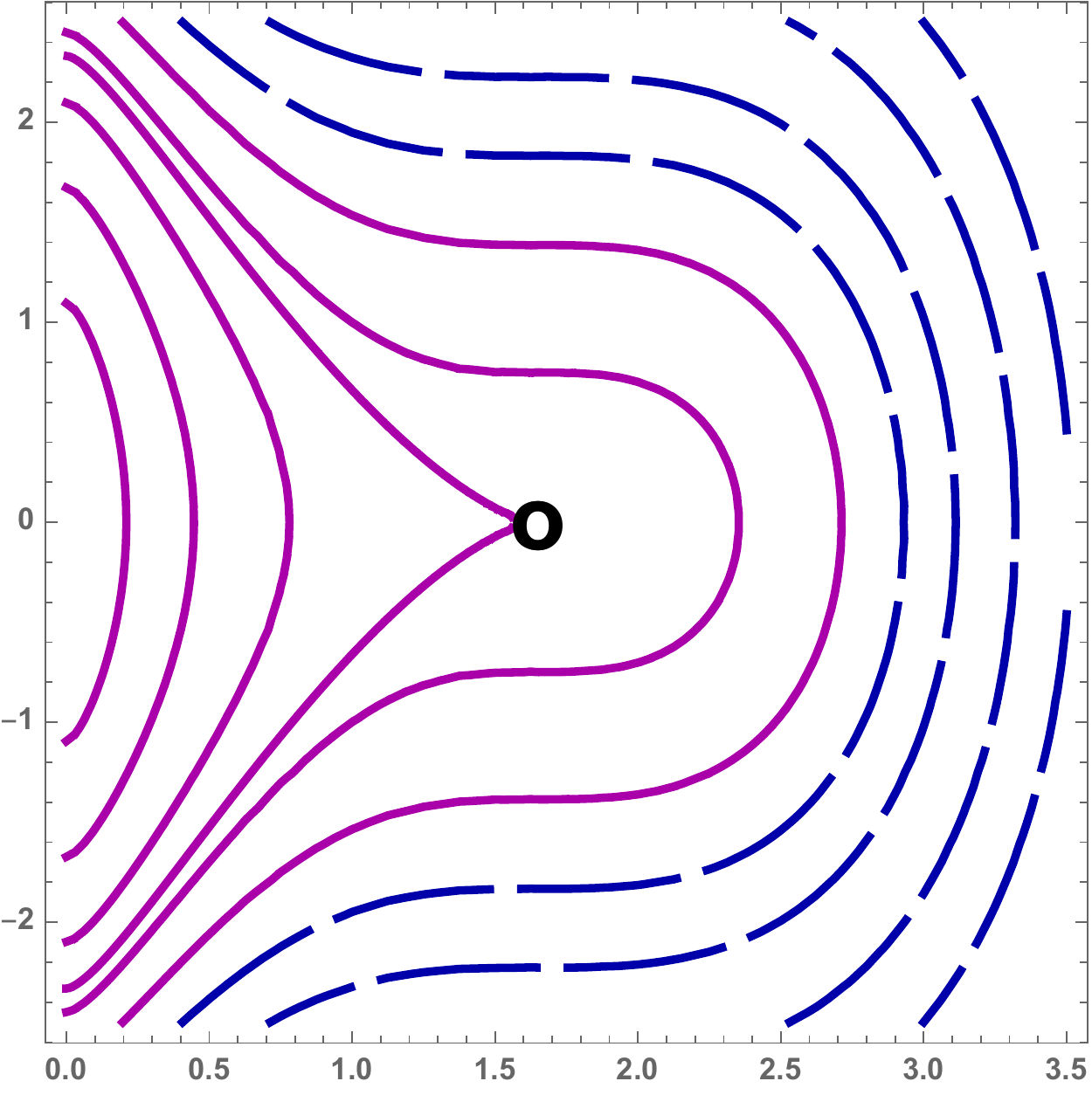}\quad\,\includegraphics[width=.31\textwidth]{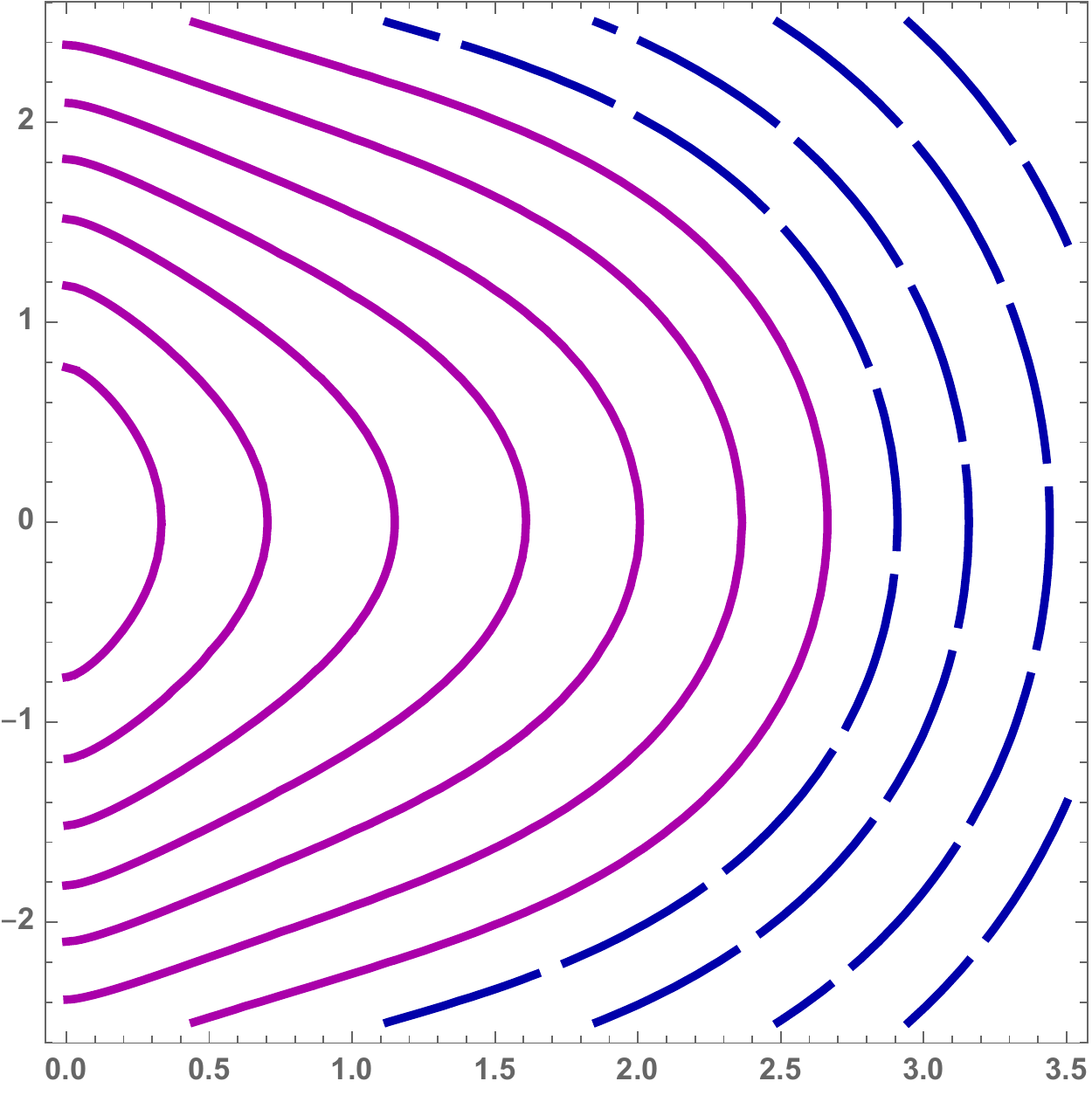}
\end{center}
\caption{Orbits in phase planes. (Left) Case $4\rho\mu^2<1$: here $\rho=0$, $\mu=1$ and the critical points are $P_{-}=(1,0)$ and $P_{+}=(e,0)$, respectively. (Centre) Case $4\rho\mu^2=1$: here $\rho=1$, $\mu=1/2$ and the critical points are $P_{-}=P_{+}=(\sqrt{e},0)$. (Right) Case $4\rho\mu^2>1$: here $\rho=1$, $\mu=1$ and there are no critical points.}
\label{orbitas}
\end{figure}

For each fixed connected component of an orbit, we call $x_0$ to the maximum value of $x>0$. From \eqref{F},  this value $x_0$ is reached when $y=0$. Define the function 
$$\widehat{F}(x)=F(x,0)=x^2\left(\left(\log x-1\right)^2+\rho\mu^2\right)\,.$$
When $x\rightarrow 0$, $\widehat{F}(x)$ tends to $0$. Then, $\widehat{F}(x)$ increases in the interval $\left(0,x_{-}\right)$. After that value, $\widehat{F}(x)$  begins to decrease until it reaches the other value $x_{+}$ and, finally,  $\widehat{F}(x)$ increases until infinity. Moreover, the value of $\widehat{F}(x)$ at $x_{\pm}$ is
$$\widehat{F}(x_{\pm})=x_{\pm}^2\log x_{\mp}\,.$$
This implies that for each fixed $d>0$ the equation $ \widehat{F}(x)=d$ has either one or three solutions depending on the value of $\rho$.

With the notation introduced in \eqref{not}, the critical curve $\gamma$ can be arc-length parametrized as 
\begin{equation}\label{param}
\gamma(s)=\Phi\left(\mu x(s),\psi(s)\right),
\end{equation}
where  the function $\Phi:\mathbb{R}^2\rightarrow \m^2(\rho)\subset\mathbb{R}^3$ is
\begin{eqnarray*}
 \Phi(u,v)=\begin{cases} \dfrac{1}{\sqrt{d}}\left(u,dv,0\right), & \mbox{if } \rho=0\, \\ \dfrac{1}{\sqrt{\rho\,d}}\left(\sqrt{\rho}\,u,\sqrt{d-\rho u^2}\sin\left(\sqrt{\rho d}v\right),\sqrt{d-\rho u^2}\cos\left(\sqrt{\rho d}v\right)\right), & \mbox{if } \rho>0
  \\ \dfrac{1}{\sqrt{-\rho\,d}}\left(\sqrt{-\rho}\,u,\sqrt{d-\rho u^2}\sinh\left(\sqrt{-\rho d}v\right),\sqrt{d-\rho u^2}\cosh\left(\sqrt{-\rho d}v\right)\right), & \mbox{if } \rho<0\,, \end{cases}
\end{eqnarray*}
and the function $\psi(s)$ is determined by the arc-length parametrization. We   can assume that the arc-length parameter is chosen so that $x(0)=x_0$, where recall that $x_0$ is the maximum value of $x>0$ in each fixed orbit. Then 
\begin{equation}\label{psi}
\psi(s)=\int_0^s \frac{x(t)\left(\log x(t)-1\right)}{d-\rho\mu^2x(t)^2}\,dt\,.
\end{equation}
The function $\psi(s)$ is well defined provided that $d-\rho\mu^2x(t)^2\not=0$. We observe that $d=\rho\mu^2x(t)^2$ only when $\rho>0$ and the curve passes through the pole. 

\begin{proposition}\label{pole} Let $\gamma$ be a critical curve in $\mathbb{S}^2(\rho)$  for the energy $\mathbf{\Theta}_\mu$. If  $d>0$ is the constant of integration in \eqref{fi}, then $\gamma$ passes through the pole of $\mathbb{S}^2(\rho)$ if and only if $d=\rho\,\mu^2e^2$ and $x_0\geq e$.
\end{proposition}

\begin{proof} 
By   \eqref{F},   $d=\rho\mu^2x(t)^2$ at the point $P=(x,y)$ if and only if $\mu^2y^2+x^2\left(\log x-1\right)^2=0$.   Thus   $P=(e,0)$. Now,   the parametrization of $\gamma$ in \eqref{param}  concludes the result. 
\end{proof}

We parametrize the critical curve $\gamma$ in terms of the variable $x$. Using \eqref{F}, the change of variable $y=x_s$ in the expression \eqref{psi} of $\psi(s)$ gives
\begin{equation}\label{Psi(x)}
\psi\left(x\right)=\int_{x_0}^{x}\frac{r\left(\log r-1\right)}{\left(d-\rho\mu^2r^2\right)y(r)}\,dr=\mu\int_{x_0}^{x}\frac{r\left(\log r-1\right)}{\left(d-\rho\mu^2r^2\right)\sqrt{d-\rho\mu^2 r^2-r^2\left(\log r-1\right)^2}}\,dr\,.
\end{equation}
It turns out that for any $\rho$, $d>\rho\mu^2x^2+x^2\left(\log x-1\right)^2>\rho\mu^2x^2$ holds almost everywhere. Then $\psi'(x)<0$ if $x\in\left(0,e\right)$ and $\psi'(x)>0$ when $x>e$, provided that we are considering the symmetric  branch of the critical curve $\gamma$ given by $y(x)\geq 0$ (see Proposition \ref{sym}). This implies that $\psi(x)$ decreases while $x\in\left(0,e\right)$ and increases for $x>e$. 

From the parametrization \eqref{param} and using \eqref{Psi(x)}, we now  deduce some properties about the symmetries of $\gamma$ and its cuts with respect to some coordinate axis.  The following result will simplify the work of the classification of Section \ref{sec4}. 

\begin{proposition}\label{sym} Let $\gamma$ be a critical curve of $\mathbf{\Theta}_\mu$ in ${\mathbb M}^2(\rho)$. Then, up to rigid motions, we have: 
\begin{enumerate}
\item The curve $\gamma$ is symmetric with respect to the geodesic $\alpha= \Pi_{13}\cap {\mathbb M}^2(\rho)$, where $\Pi_{13}$ is the plane of equation $x_2=0$. 
\item When $x$ goes to $0$, $\gamma$ tends to cut the geodesic $\beta=\Pi_{23}\cap {\mathbb M}^2(\rho)$ orthogonally, where $\Pi_{23}$ is the plane of equation $x_1=0$. Moreover, the geodesic $\beta$ corresponds with the rotation axis $\zeta$.
\end{enumerate}
\end{proposition}
\begin{proof}
After suitable rigid motions, we assume that $\gamma$  is parametrized by \eqref{param}. 
\begin{enumerate}
\item Consider the change of variable   in \eqref{psi} given by $\psi(-s)=-\psi(s)$. From the parametrization \eqref{param}, the second coordinate of $\gamma(-s)$ has the opposite sign to the corresponding one of $\gamma(s)$ and the first and third coordinates coincide. Thus, $\gamma$ is symmetric with respect to $\alpha$. This proves the first statement. 
\item From the parametrization \eqref{param}, $\gamma$ would meet the geodesic $\beta$ if its first component vanishes, that is, when $x=e^{\,\mu\kappa}=0$, which is outside of the parametrization. However, in order to obtain complete curves, the point $x=0$ can be added by continuity.

We now compute the first derivative of the critical curve $\gamma$ using the parametrization $\gamma(x)$, \eqref{param} and \eqref{Psi(x)}, to obtain
$$\gamma'(x)=\mu \Phi_u\left(\mu x,\psi(x)\right)+\psi'(x)\Phi_v\left(\mu x,\psi(x)\right).$$
As $x\rightarrow 0$, from \eqref{Psi(x)} we have $\psi(x)\rightarrow 0$ and $\psi'(x)\rightarrow 0$. Then a simple computation from \eqref{param} concludes that $\gamma'(x)\rightarrow  \left(\mu/\sqrt{d},0,0\right)$
as $x\rightarrow 0$. Thus the tangent   of the critical curve $\gamma$ is orthogonal to the geodesic $\beta$ at $x=0$. Finally, from the geometric construction introduced above, the vector field $\mathcal{I}$ defined in  \eqref{I}, and the parametrization \eqref{param}, it is clear that the axis of rotation $\zeta$ is, precisely, the geodesic $\beta$.
\end{enumerate}
\end{proof}

\begin{remark}\label{regularity} For the critical curves $\gamma$ that  cut orthogonally the geodesic $\beta$, after reflecting them across $\beta$, we obtain   complete curves of class $\mathcal{C}^1$ on the points $x=0$. However, these curves are not of  class $\mathcal{C}^2$ in those points since the curvature $\kappa=\log\left(x\right)/\mu$ tends to $-\infty$ when $x$ goes to zero. By the variational problem, in the rest of the points, these curves are smooth. Moreover, we recall that after the binormal evolution shown in previous sections, the same regularity conditions hold for the associated rotational surface with constant skew curvature.
\end{remark}
 
\section{Classification of the critical curves in $2$-space forms}\label{sec4}

In this section we give a   classification of the critical curves for the exponential type curvature energy $\mathbf{\Theta}_\mu$ showing all possible shapes for these curves. This  will be done depending on the sign of the constant sectional curvature $\rho$: see Theorem \ref{r2} ($\rho=0$), Theorems \ref{r3} and \ref{r4} ($\rho>0$) and Theorem \ref{r5} ($\rho<0$).

\subsection{Case $\rho=0$: the Euclidean plane $\r^2$}

We   consider the Euclidean plane $\r^2$ ($\rho=0$). Firstly, we prove a result about the behavior of the critical curves by dilations. We write $F=F(x,y;\mu)$ to indicate the dependence of $F$ on the index energy $\mu$.

\begin{proposition}\label{unique} In the Euclidean plane $\mathbb{R}^2$, any dilation of ratio $\lambda>0$ of a critical curve of $\mathbf{\Theta}_\mu$ for $d>0$ is a critical curve of $\mathbf{\Theta}_{\lambda\mu}$ for $d>0$. 
\end{proposition}
\begin{proof}
Consider a dilation of ratio $\lambda>0$ of $\gamma$, say $\widetilde{\gamma}( s)=\lambda\gamma(s/\lambda)$. Then its curvature $\widetilde{\kappa}$ is $\widetilde{\kappa}( s)=\kappa(s/\lambda)/\lambda $. Moreover $x(s)=\widetilde{x}(\lambda s)$ and $y(s)=\lambda \widetilde{y}(\lambda s)$. If $\widetilde{\mu}=\lambda\mu$, this implies  $d=F(x,y;\mu)=F(\widetilde{x},\widetilde{y};\widetilde{\mu})$. 
\end{proof}

Using Propositions \ref{orientation} and   \ref{unique}, after a change of orientation  and a dilation if necessary, we suppose that the energy index is $\mu=1$. By \eqref{param} and \eqref{Psi(x)}, the parametrization of $\gamma$ in the Euclidean plane is
\begin{equation}\label{paramr2}
\gamma(x)=\frac{1}{\sqrt{d}}\left(x,\int \frac{x\left(\log x-1\right)}{\sqrt{d-x^2\left(\log x-1\right)^2}}\, dx\right).
\end{equation}

We are now in the right position to describe the critical curves of $\mathbf{\Theta}_1$ in the Euclidean plane $\r^2$. In  figure \ref{criticalR2}, we show the shapes of the critical curves and the corresponding rotational surfaces are   shown in figure \ref{surfacesR3}. Here we follow the terminology given in the Introduction. 

\begin{theorem}\label{r2} 
Critical curves with non-constant curvature of $\mathbf{\Theta}_1$ in $\r^2$ are of oval type, simple biconcave type, figure-eight type, non-simple biconcave type, borderline type and orbit-like type.
\end{theorem}

\begin{proof} The values of $x_{\pm}$ in \eqref{ll} are $x_{-}=1<e=x_{+}$ with $\widehat{F}(x_{+})=0<1=\widehat{F}(x_{-})$. Thus   the orbit $F(x,y)=d$ has exactly three cuts with the axis $y=0$ if $d<1$ and only one cut if $d> 1$. For $d=1$ we have two cuts, but one of them corresponds with the critical point $P_{-}$ whose associated critical curve is a straight line and it is out of our consideration here.

We classify the critical curves depending on the value of $d$ and the type of orbit (see figure \ref{orbitas}, left). By the symmetry proven in Proposition \ref{sym}, we only consider the branch $y\geq 0$ in the associated orbit. 
\begin{enumerate}
\item Case $d\leq \widehat{F}(x_{-})=1$ and $x_0\leq x_{-}=1$ (see purple orbits). We start at $x=0$ where the curve $\gamma$ cuts orthogonally the geodesic $\beta$ (see Proposition \ref{sym}). Then, it is clear that $x$ increases until its maximum value $x_0\leq x_{-}=1$. At the same time, the function $\psi(x)$ defined in \eqref{Psi(x)}, decreases. Moreover, at $x=x_0$ we have $\psi(x)=0$ which means that the curve cuts the symmetry axis $\alpha$. Thus, $\gamma(x)$ for these values is of oval type.
\item Case $d=\widehat{F}(x_{-})=1$ and $x_0>x_{-}=1$ (see the green orbit). In this case, $x$ is always bigger than $1$, so the curve does not cut $\beta$. When $x\rightarrow 1$, the curvature $\kappa=\log x$ tends to zero, hence  $\gamma$ is asymptotic to a straight line. This straight line is represented in the phase plane by the critical point $P_{-}$. Moreover, differentiating $\gamma(x)$ in \eqref{paramr2}, we see that the tangent of $\gamma$ tends to be vertical when $x\rightarrow 1$. Thus, $\gamma$ is asymptotic to a vertical line. Now, an analysis of the function $\widehat{F}(x)$ tells us that, indeed, $x_0>e$. Then, while $x$ increases from $1$ to $x_0$, the function $\psi(x)$  decreases until $x=e$ and then it increases. Again, at $x=x_0$ we have $\psi(x)=0$. Finally, notice that, by symmetry the curve is non-simple since it must cut the axis $\alpha$ before $x_0$. Therefore, we obtain that $\gamma(x)$ is of borderline type.
\item Case $d>\widehat{F}(x_{-})=1$ (see blue orbits). We start at $x=0$, where the curve meets $\beta$ orthogonally. Then, while $x$ increases until $x_0$, the function $\psi(x)$ in \eqref{Psi(x)}, decreases from $x=0$ to $x=e$ and then it increases. We may have three different types of curves depending on the sign of $\psi(0)$:
\begin{enumerate}
\item If $\psi(0)>0$, due to the behavior of $\psi(x)$,  apart from $x_0$ there must be another cut with $\alpha$ somewhere between $0$ and $x_0$. That is, $\gamma(x)$ is non-simple. Hence, $\gamma(x)$ is of non-simple biconcave type.
\item If $\psi(0)=0$, this means that $\gamma(x)$ cuts $\alpha$ exactly at $x=0$ and $x=x_0$, which gives rise to the figure-eight type curve.
\item Finally, for $\psi(0)<0$, then the curve $\gamma(x)$ only meets the axis $\alpha$ at $x_0$. This case corresponds with  simple biconcave type.
\end{enumerate}
\item Case $0=\widehat{F}(x_{+})<d<\widehat{F}(x_{-})=1$ and $x_0>x_{+}=e$ (see brown orbits). The orbits for these values are closed and, hence, the curvature $\kappa$ of $\gamma(x)$ is periodic. The possible values of $x$ are restricted to lie on an interval, which means that the critical curve is bounded by two vertical lines. Moreover, since $x>1$ necessarily,  $\kappa=\log x>0$ and, therefore, the curve is of orbit-like type. 
\end{enumerate}
This concludes the proof.
\end{proof}

\begin{figure}[hbtp]
\begin{center}
\includegraphics[width=.16\textwidth]{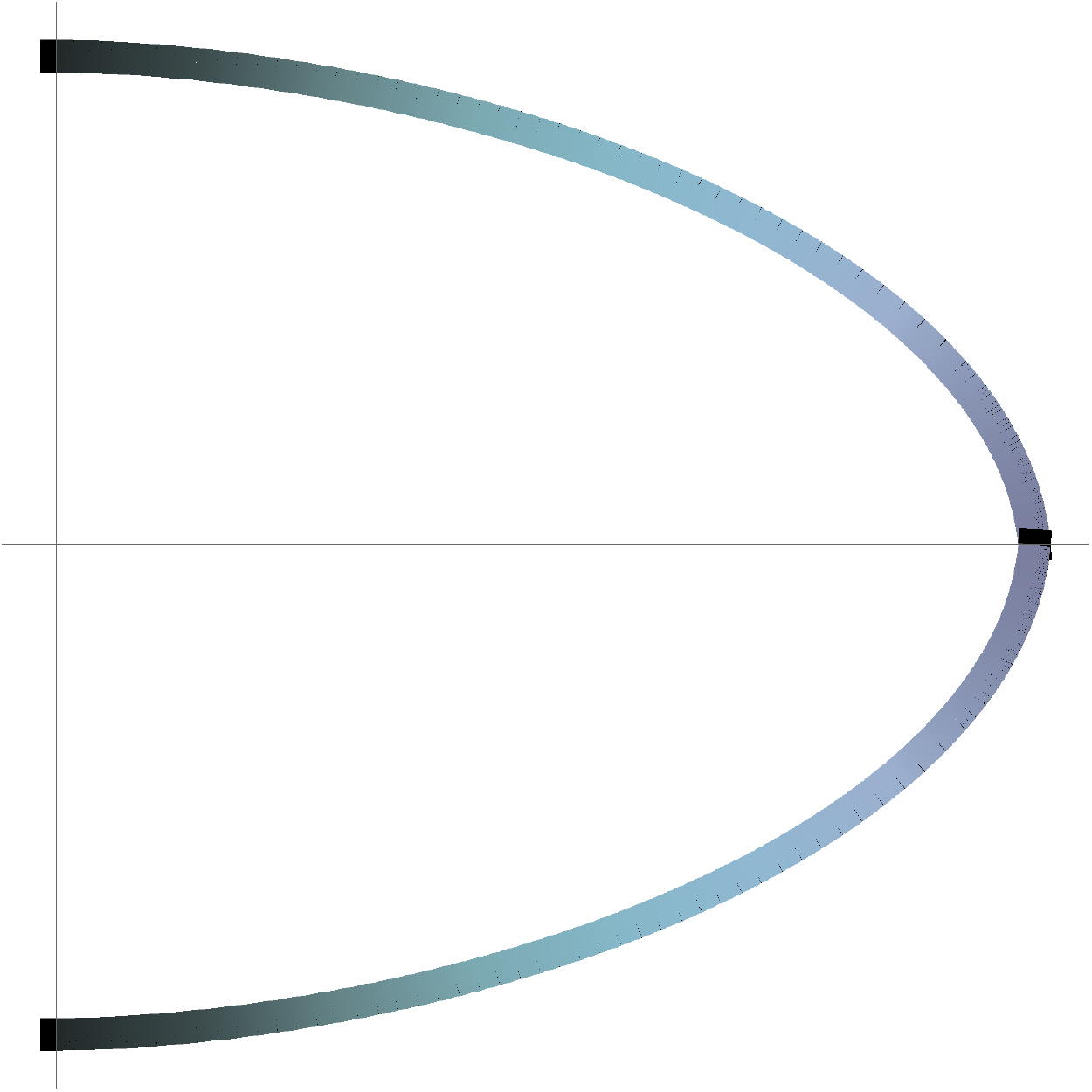}\,\includegraphics[width=.16\textwidth]{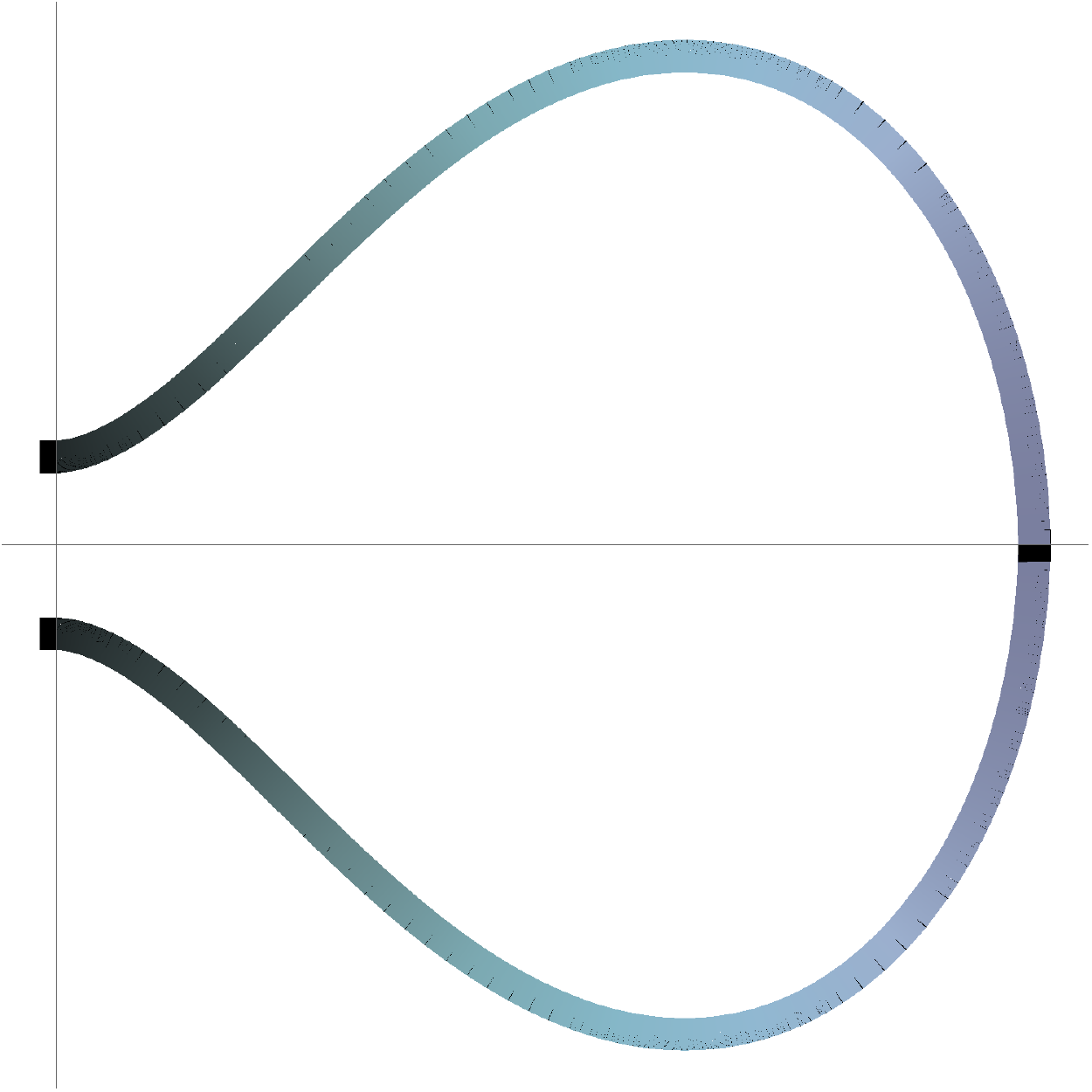}\,\includegraphics[width=.16\textwidth]{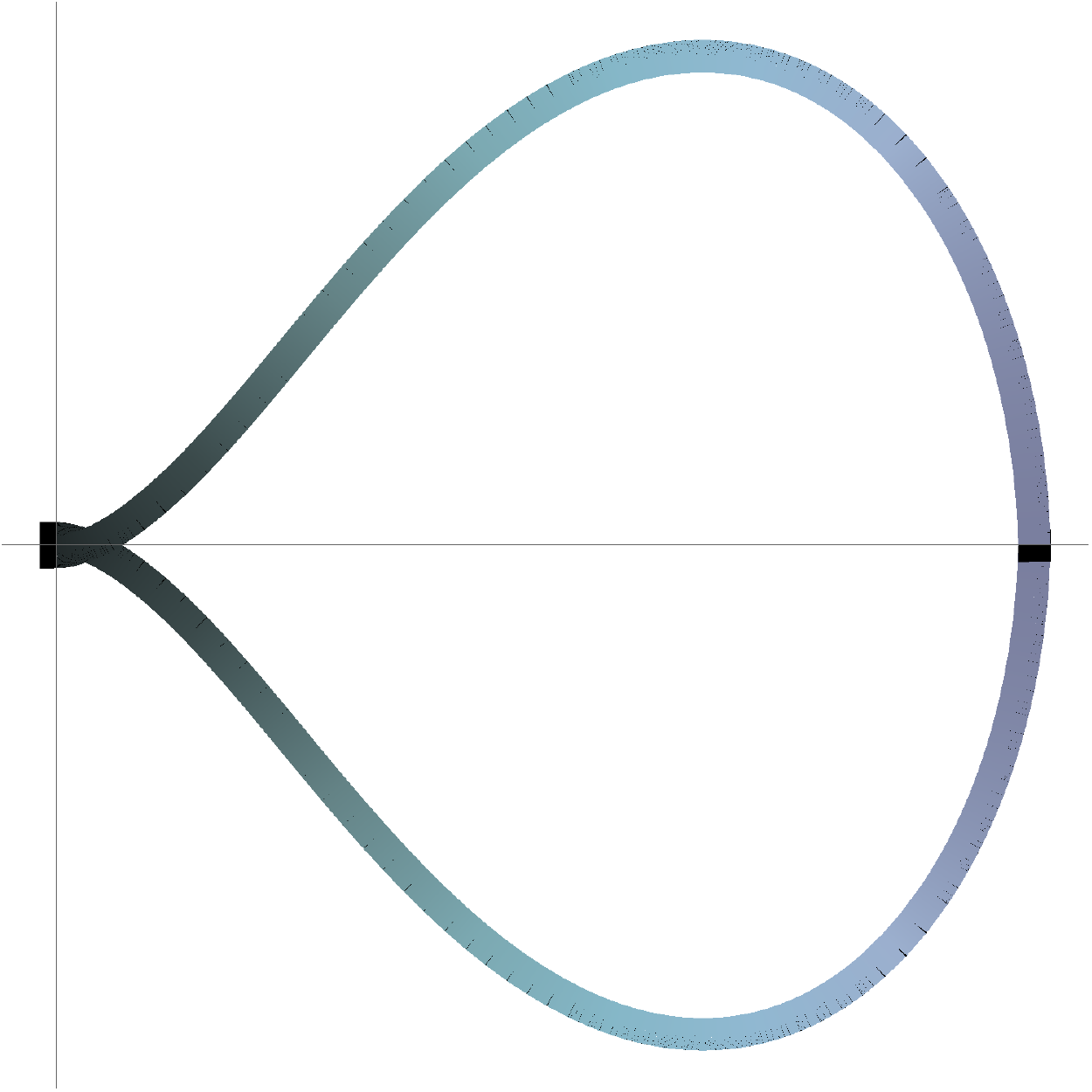}\,\includegraphics[width=.16\textwidth]{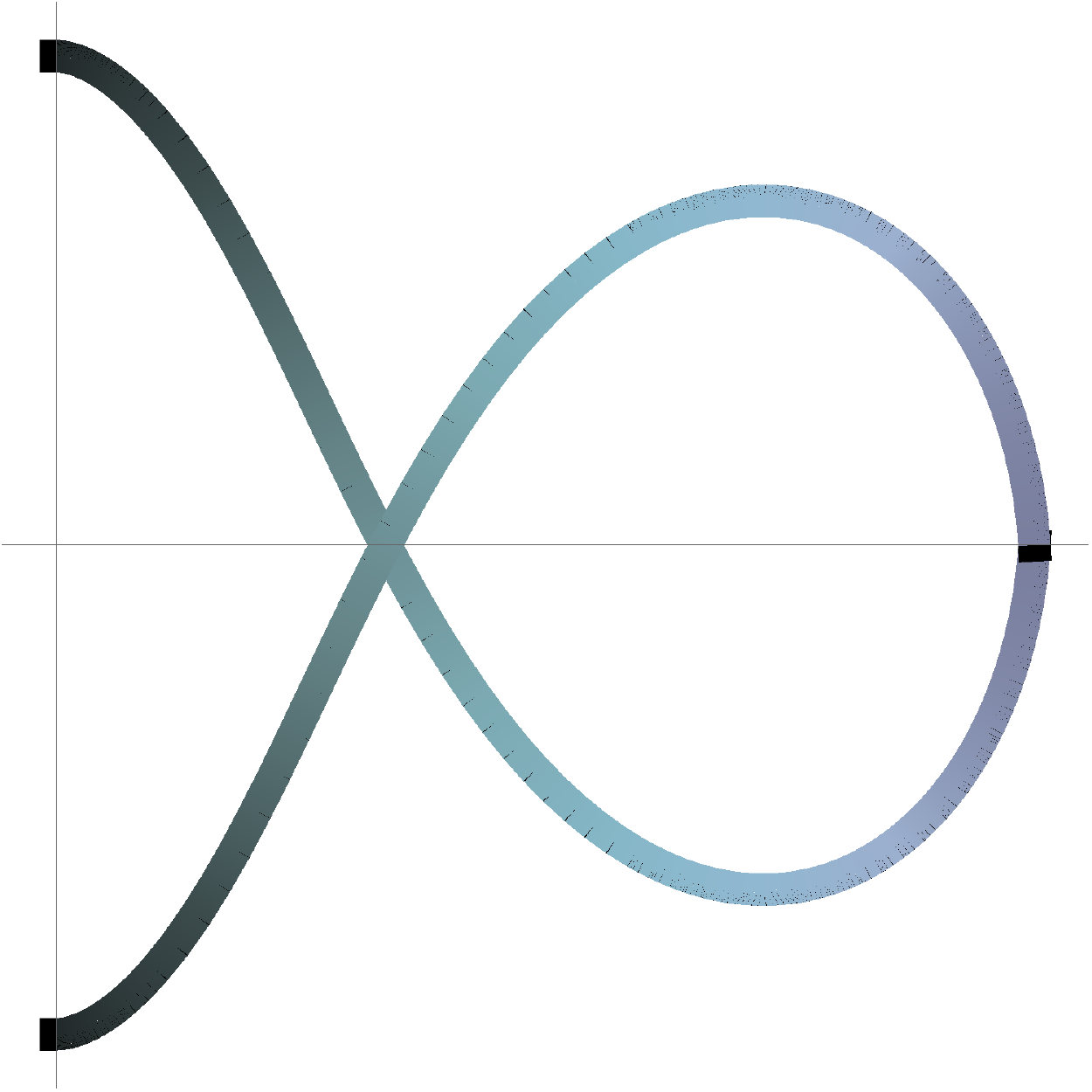}\, \includegraphics[width=.16\textwidth]{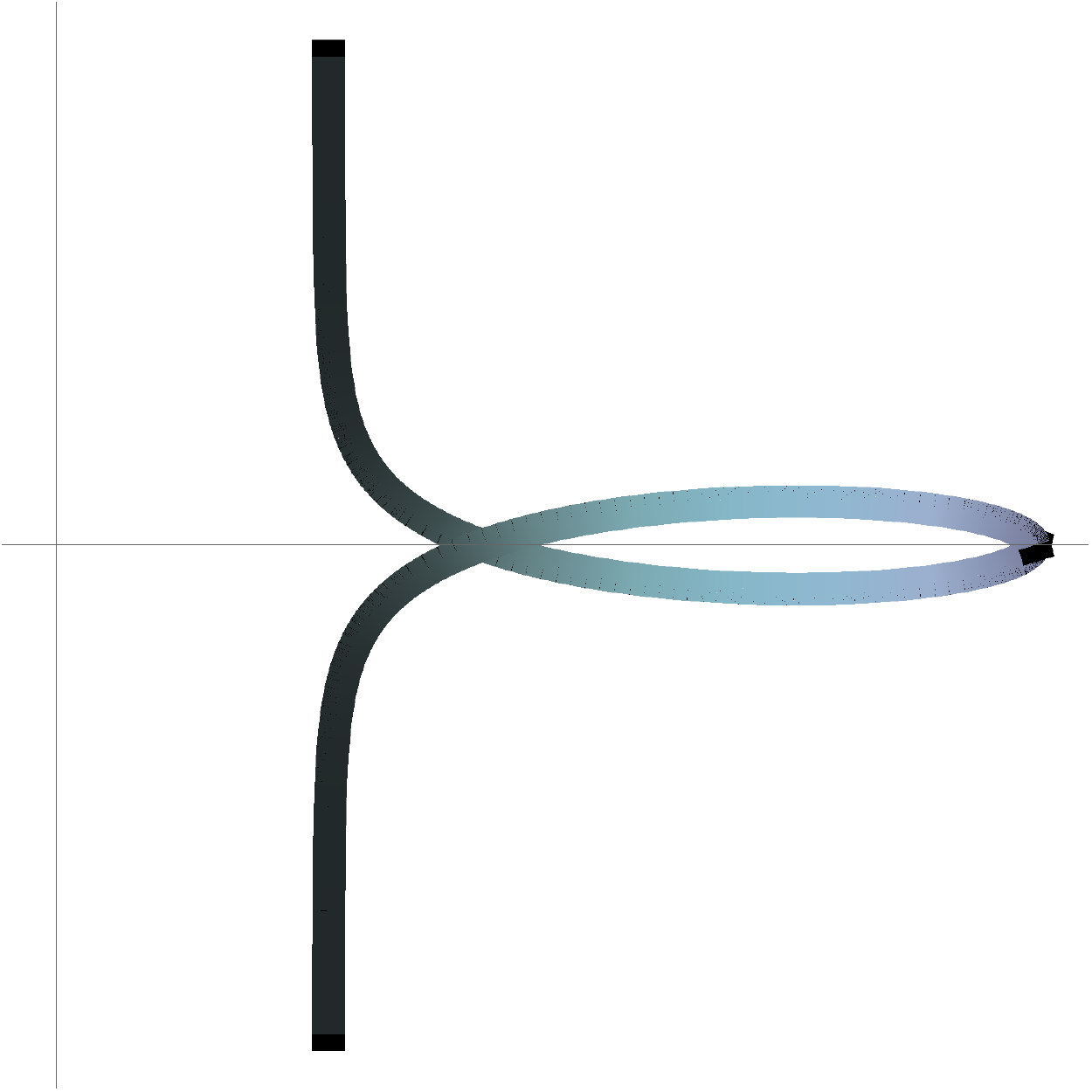}\,\includegraphics[width=.16\textwidth]{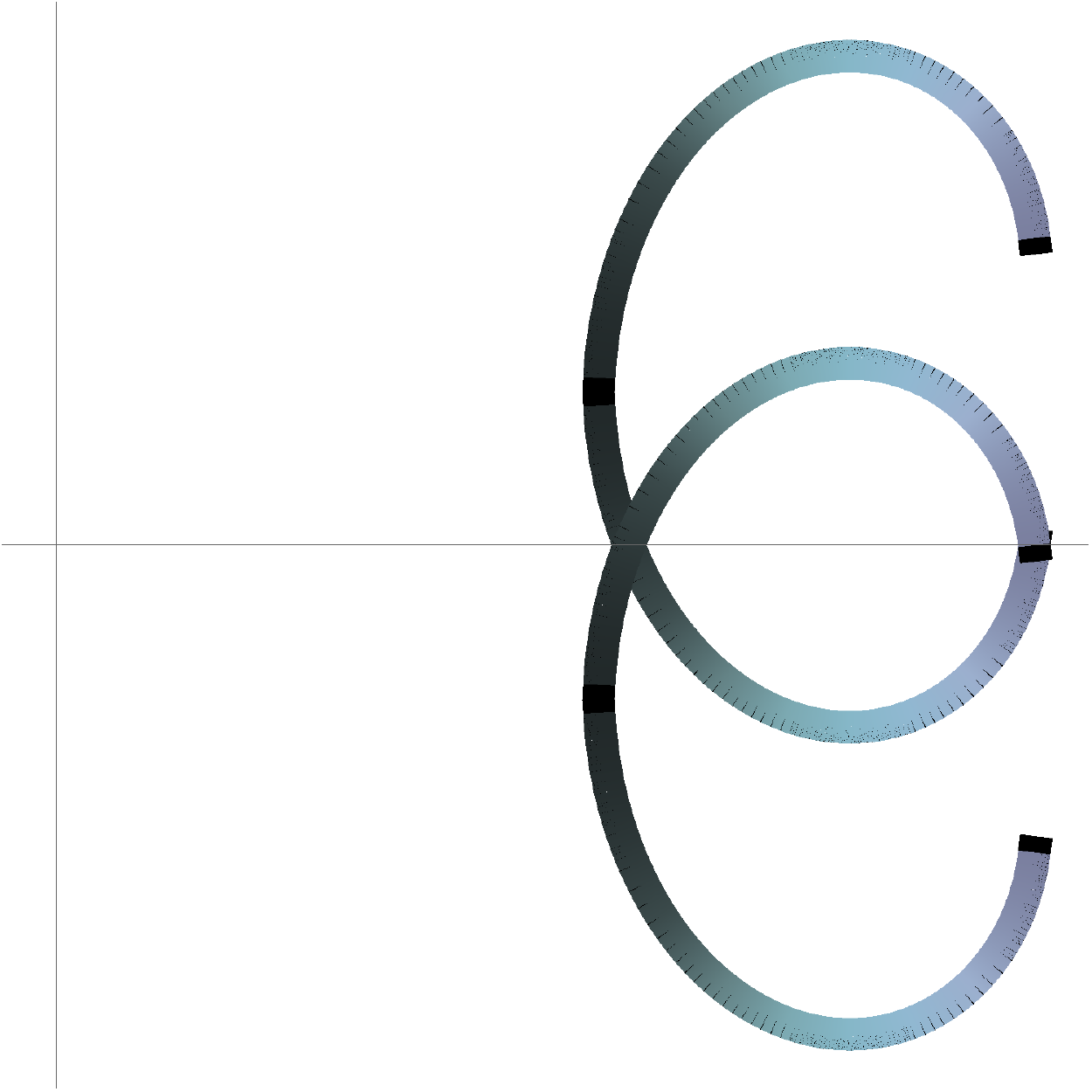}
\end{center}
\caption{Critical curves of $\mathbf{\Theta}_1$ with non-constant curvature in the Euclidean plane $\r^2$. The geodesic $\alpha$ is represented by the horizontal axis, whereas $\beta$ is the vertical one. From left to right:  oval type ($d=0.5$),  simple biconcave type ($d=3.5$), figure-eight type ($d\simeq 2.9$),  non-simple biconcave type ($d=1.5$), borderline type ($d=1$) and orbit-like type ($d=0.5$).}
\label{criticalR2}
\end{figure}

\subsection{Case $\rho>0$: the sphere $\s^2(\rho)$}

 After rescaling, we may assume that $\rho=1$, so we are considering the sphere of radius one  $\s^2(1)$. Recall also that  the energy index is   positive (Proposition \ref{orientation}). In this setting, the parametrization \eqref{param}  of $\gamma$ is
\begin{equation}\label{params2}
\gamma(x)=\frac{1}{\sqrt{d}}\left(\mu x,\sqrt{d-\mu^2 x^2}\sin\left(\sqrt{d}\,\psi(x)\right),\sqrt{d-\mu^2x^2}\cos\left(\sqrt{d}\,\psi(x)\right)\right)
\end{equation}
where the function $\psi(x)$ is given in \eqref{Psi(x)}.

The classification of the critical curves of $\mathbf{\Theta}_\mu$ depends on the value of $\mu>0$. First we consider the case $\mu<1/2$: see figure \ref{criticalS2}. 

\begin{theorem}\label{r3} 
If  $\mu<1/2$, then  the critical curves of $\mathbf{\Theta}_\mu$ in $\s^2(1)$ with non-constant curvature are of oval type, simple biconcave type, figure-eight type, non-simple biconcave type, borderline type and orbit-like type.
\end{theorem}
\begin{proof} If $\mu<1/2$, then  $1<x_{-}<\sqrt{e}< x_{+}<e$ and $0<\widehat{F}(x_{+})<\widehat{F}(x_{-})$. Then, the orbit $F(x,y)=d$ has exactly three cuts with the axis $y=0$ if $\widehat{F}(x_{+})<d<\widehat{F}(x_{-})$. Otherwise, we have only one cut. Again, note that for $d=\widehat{F}(x_{-})$ we have two cuts, but one of them corresponds with the critical point $P_{-}$ whose associated critical curve has constant curvature and, therefore, it is out of our consideration here.

Now, we classify the critical curves depending on the value of $d$ and the type of orbit, which can be described by the value $x_0$. We argue as in Theorem \ref{r2} and, by Proposition \ref{sym}, we only consider the branch $y\geq 0$ of each orbit. There are four different cases (see figure \ref{orbitas}, left):
\begin{enumerate}
\item Case $d\leq\widehat{F}(x_{-})$ and $x_0\leq x_{-}$ or $d\in\left(\widehat{F}(x_{-}),\widehat{F}(e)\right]$ (see purple orbits). The parameter of the curve increases from $x=0$ (at this point, $\gamma$ meets orthogonally the geodesic $\beta$) to $x=x_0$ (where $\gamma$ cuts the geodesic $\alpha$). At the same time, the function $\psi(x)$, \eqref{Psi(x)}, which represents the variation angle of $\gamma(x)$ (see \eqref{params2}) decreases. Therefore, the curve is of oval type.
\item Case $d=\widehat{F}(x_{-})$ and $x_0>x_{-}$ (see the green orbit). The curve $\gamma$ is asymptotic in its end-points to the critical circle represented by $P_{-}$. On other points, the $x_1$ component of \eqref{params2} is even further to the geodesic $\beta$, so $\gamma$ never cuts it. The behavior of the variation angle $\psi(x)$, \eqref{Psi(x)}, while $x$ increases from $x_{-}$ to $x_{0}$ is as follows: it decreases until $x=e$ and then increases. Also $\psi(x_{0})=0$ so that $\gamma$ cuts $\alpha$ at $x=x_0$. Due to the evolution of the variation angle, $\gamma$ cuts $\alpha$ in another intermediate point producing, after symmetry, a single loop, i.e. $\gamma$ is of borderline type.
\item Case $d>{\rm max}\{\widehat{F}(x_{-}),\widehat{F}(e)\}$ (see blue orbits). The curve $\gamma(x)$ at $x=0$ cuts the geodesic $\beta$ orthogonally. The function $\psi(x)$, \eqref{Psi(x)}, behaves as in previous point as $x$ increases. Therefore, depending on the sign of $\psi(0)$, we have:
\begin{enumerate}
\item If $\psi(0)>0$, then $\gamma$ cuts the geodesic $\alpha$ at $x=x_0$ and in another intermediate point in $\left(0,x_0\right)$, giving rise to a curve of  non-simple biconcave type.
\item If $\psi(0)=0$, then $\gamma$ cuts $\alpha$ at, precisely, $x=0$ and $x=x_0$, i.e. it is a figure-eight type curve.
\item If $\psi(0)<0$, then the only point where $\gamma$ cuts $\alpha$ is at $x=x_0$. Thus, the curve is simple. In fact, we have a  simple biconcave type curve.
\end{enumerate}
\item Case $\widehat{F}(x_{+})<d<\widehat{F}(x_{-})$ and $x_0>x_{+}$ (see brown orbits). The closure condition of the orbits in this case implies that the curvature $\kappa=\log\left(x\right)/\mu$ of $\gamma(x)$ is periodic. On the other hand, since $x>x_{-}$ for these orbits, we get that $\kappa>0$ so that $\gamma$ has no inflection points. Moreover, the same lower bound for the parameter $x$ implies that the curve lies in the smallest spherical cap whose boundary is the critical circle represented by $P_{-}$. In conclusion, $\gamma$ is of orbit-like type. 
\end{enumerate}
This concludes the proof.
\end{proof}

\begin{figure}[hbtp]
\begin{center}
\includegraphics[width=.16\textwidth]{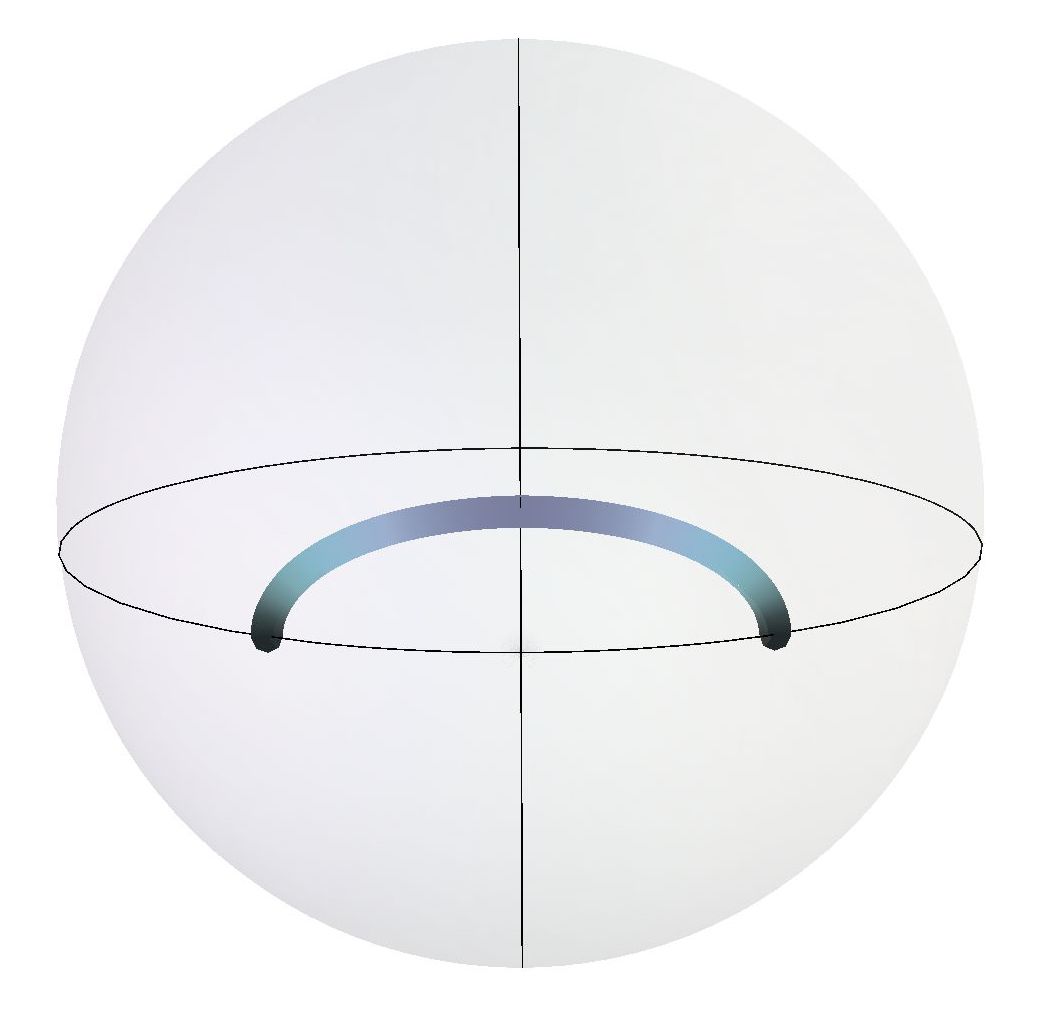}\,\includegraphics[width=.16\textwidth]{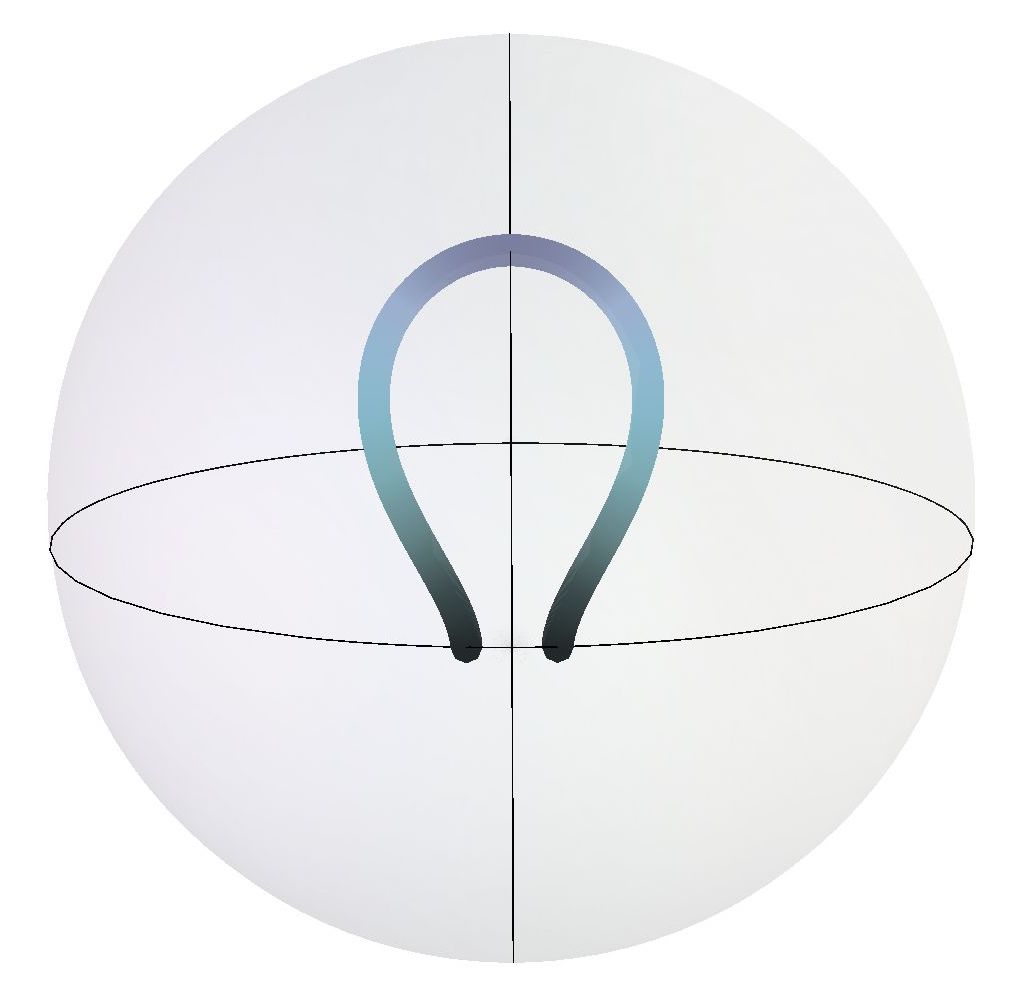}\,\includegraphics[width=.165\textwidth]{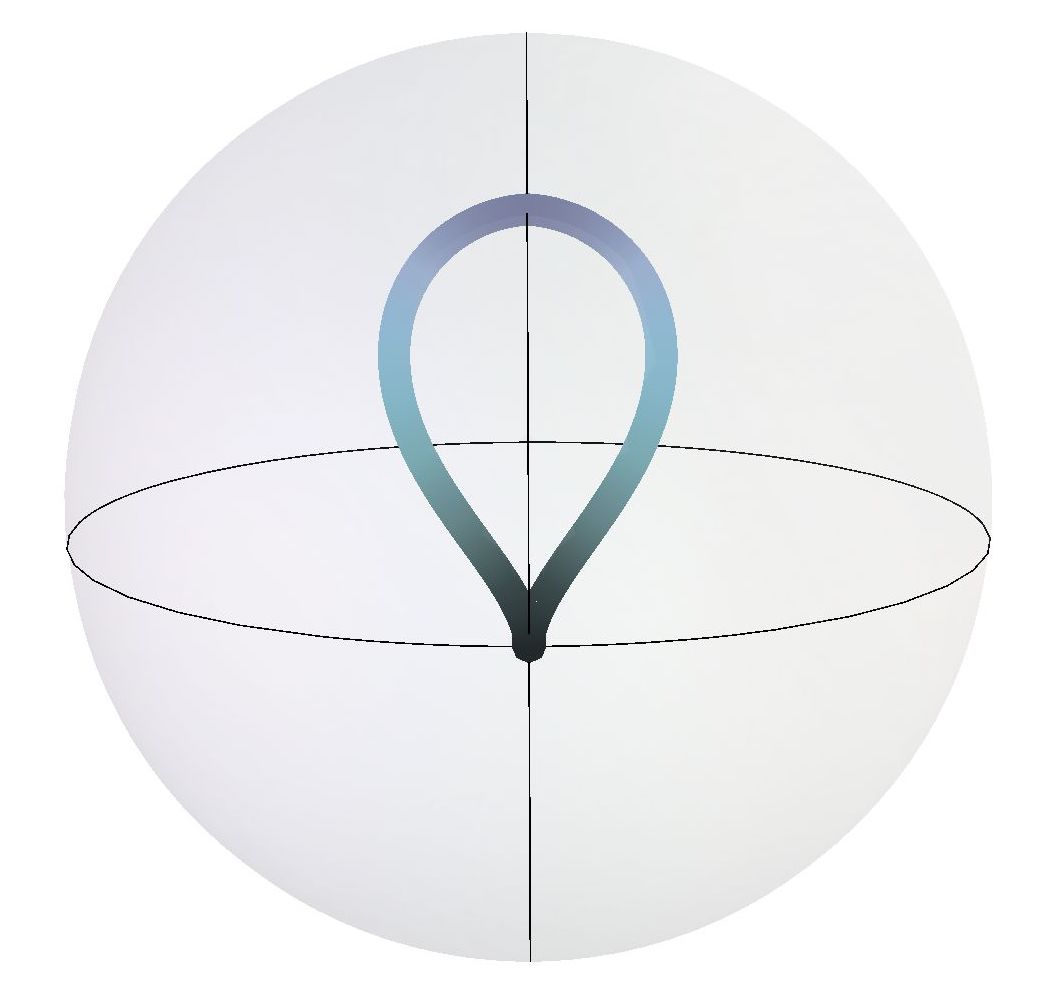}\,\includegraphics[width=.162\textwidth]{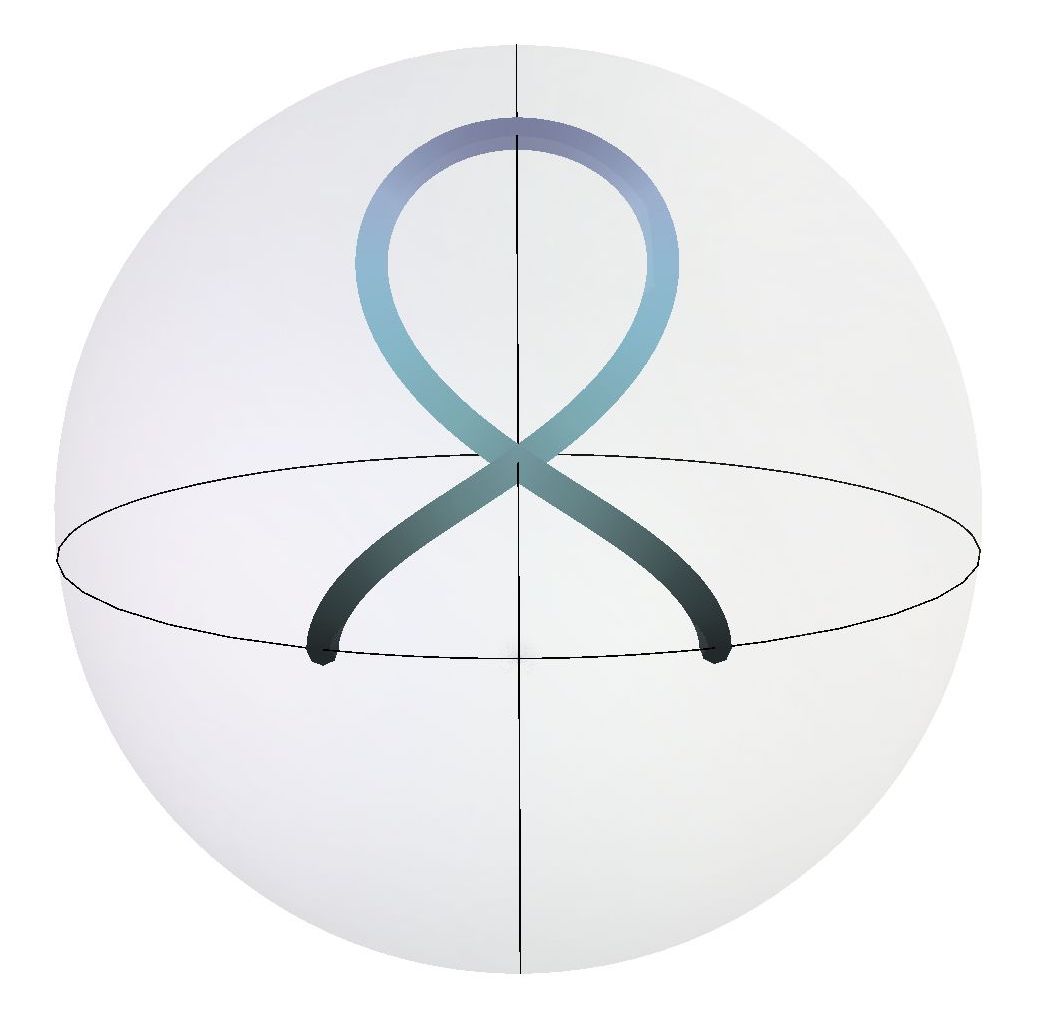}\,\,\includegraphics[width=.155\textwidth]{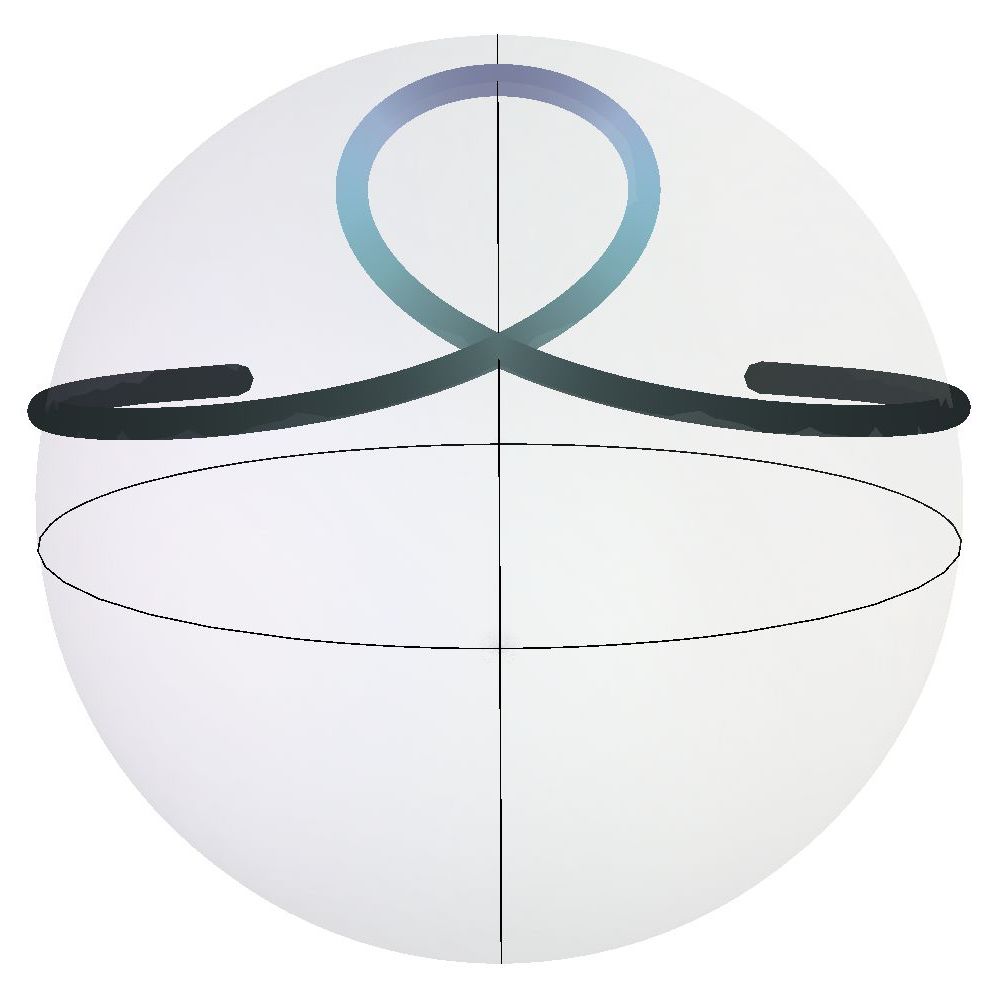}\,\,\includegraphics[width=.16\textwidth]{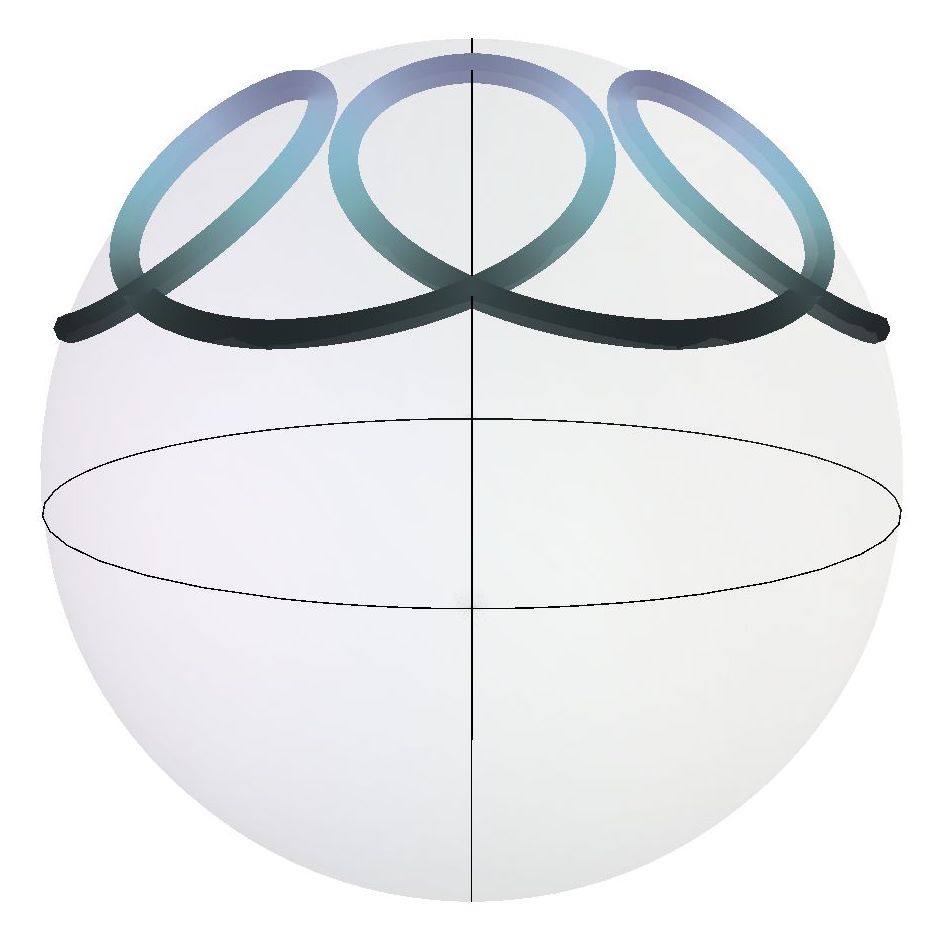}
\end{center}
\caption{Critical curves of $\mathbf{\Theta}_\mu$ with non-constant curvature in the sphere $\s^2(1)$. The geodesic $\alpha$ is represented by the meridian, whereas $\beta$ is the equator. Here $\mu=1/4$ and, from left to right:  oval type  ($d=0.95$),   simple biconcave type  ($d=4$),  figure-eight type ($d\simeq 2.9$),  non-simple biconcave type  ($d=1.5$), borderline type ($d\simeq 1.067$) and orbit-like type ($d=0.95$).}
\label{criticalS2}
\end{figure}

We now consider the case $\mu\geq 1/2$: see figure \ref{criticalS2} again. 

\begin{theorem}\label{r4}
The critical curves with non-constant curvature of $\mathbf{\Theta}_\mu$ for $\mu\geq 1/2$ in $\s^2(1)$ are of oval type, simple biconcave type, figure-eight type and non-simple biconcave type.
\end{theorem}
\begin{proof} In this case, either there are no critical points of $F(x,y)$ or the critical point is degenerate (this case appears if  $\mu=1/2$). As a consequence, the function $\widehat{F}(x)$ is monotone and the orbits $F(x,y)=d$ have only one cut each with the axis $y=0$. See figure \ref{orbitas}, center and right. Therefore, only cases 1 and 3 of Theorem \ref{r3} can occur, drawing the result.
\end{proof}

\begin{remark} By Proposition \ref{pole}, the  critical curves of $\mathbf{\Theta}_\mu$ in $\s^2(1)$ pass through the pole if and only if $d=\mu^2e^2$ and $x_0\geq e$. If $\mu< 1/2$, this can happen in critical curves of  non-simple biconcave type, borderline type and orbit-like type. If $\mu\geq 1/2$, then it can happen in critical curves of  simple biconcave type, figure-eight type and oval type. See figure \ref{criticalS2Pole}.
\end{remark}

\begin{figure}[hbtp]
\begin{center}
\includegraphics[width=.16\textwidth]{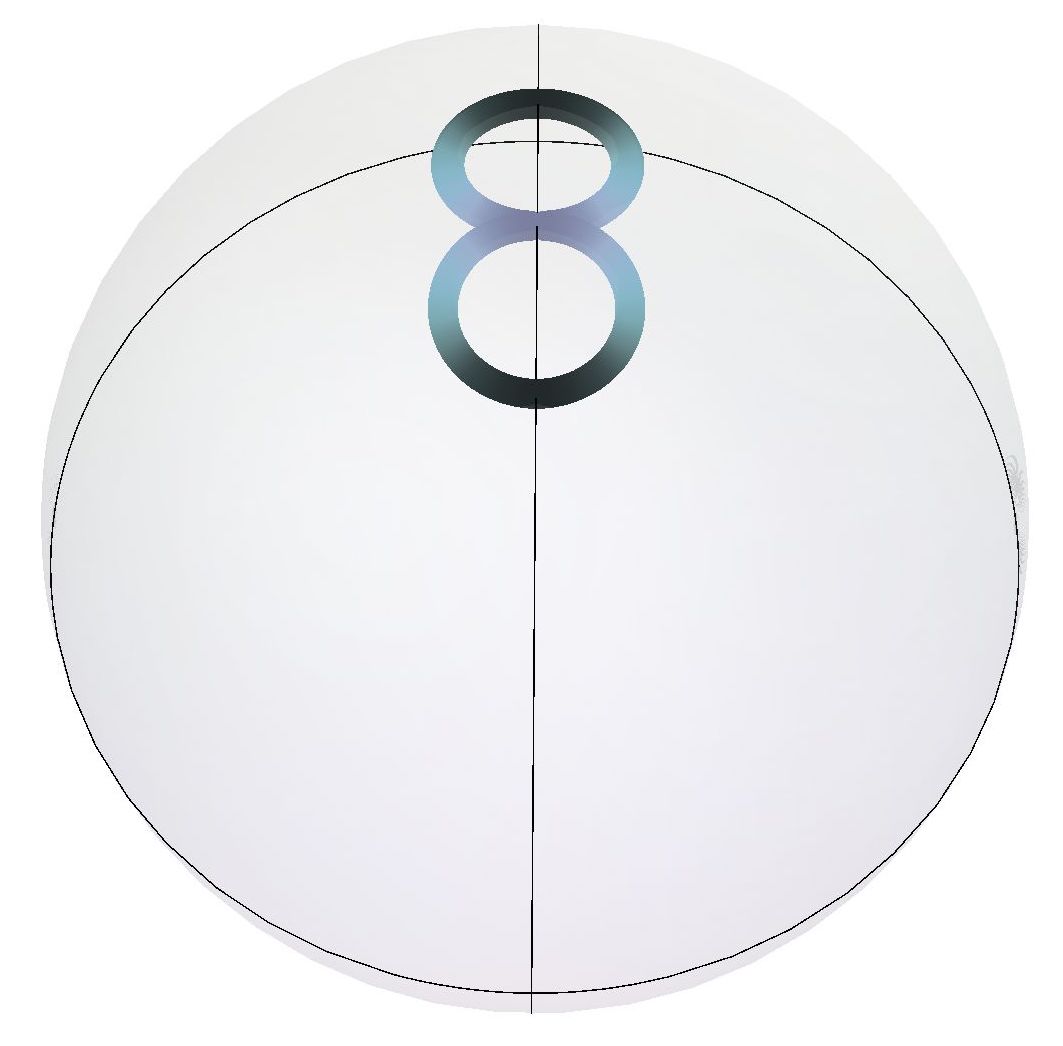}\quad\quad\includegraphics[width=.16\textwidth]{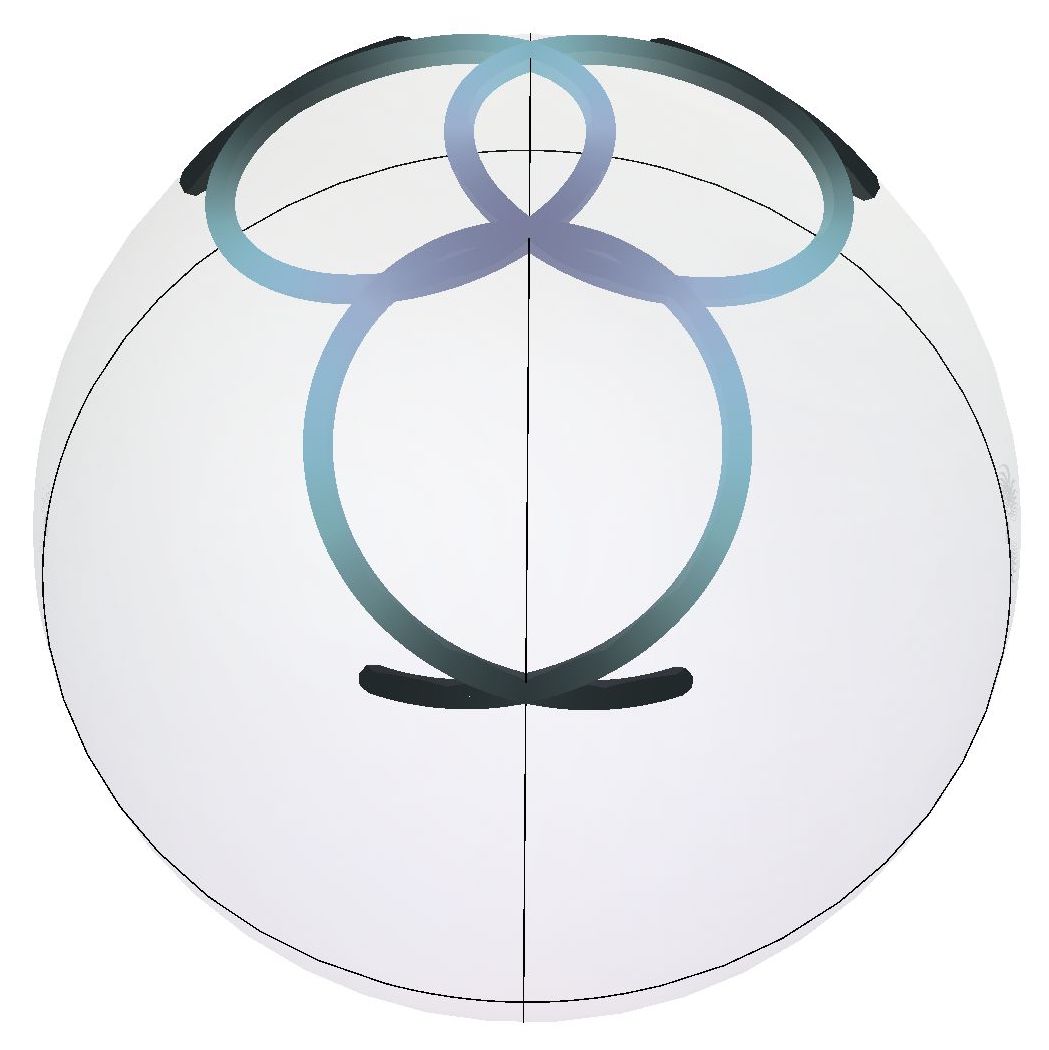}\quad\quad\includegraphics[width=.16\textwidth]{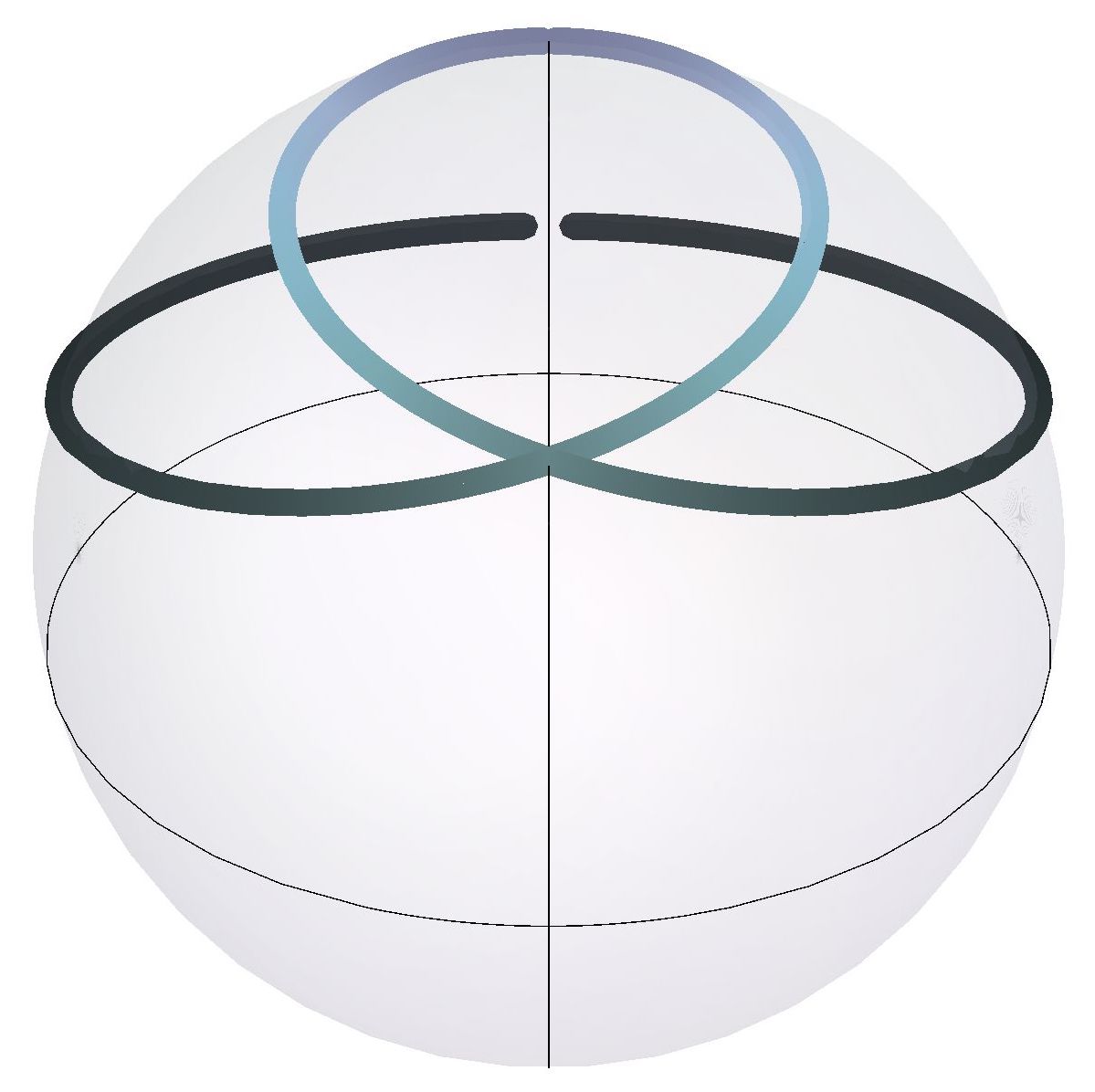}\quad\quad\includegraphics[width=.164\textwidth]{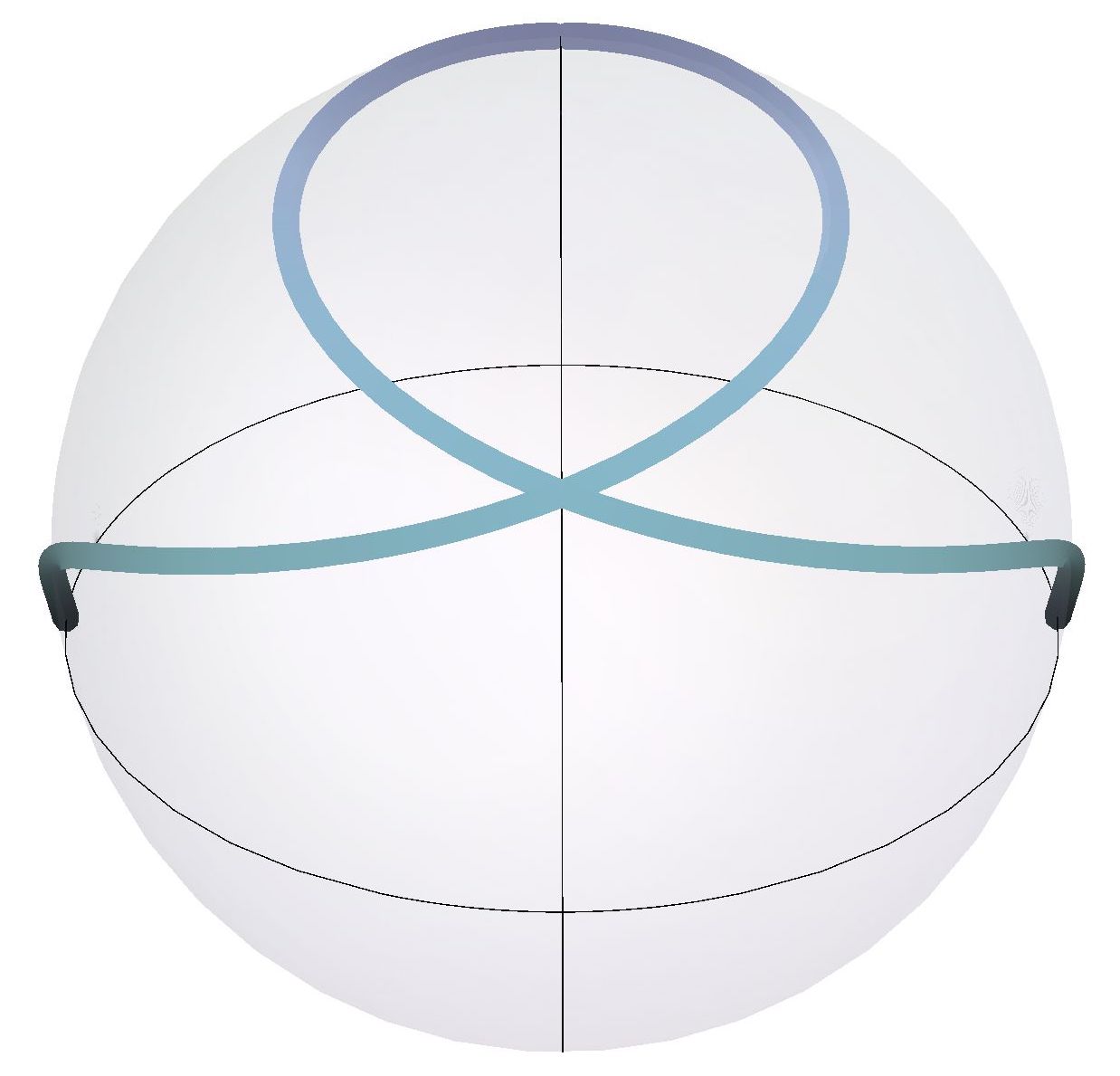}\\\vspace{0.5cm}\includegraphics[width=.16\textwidth]{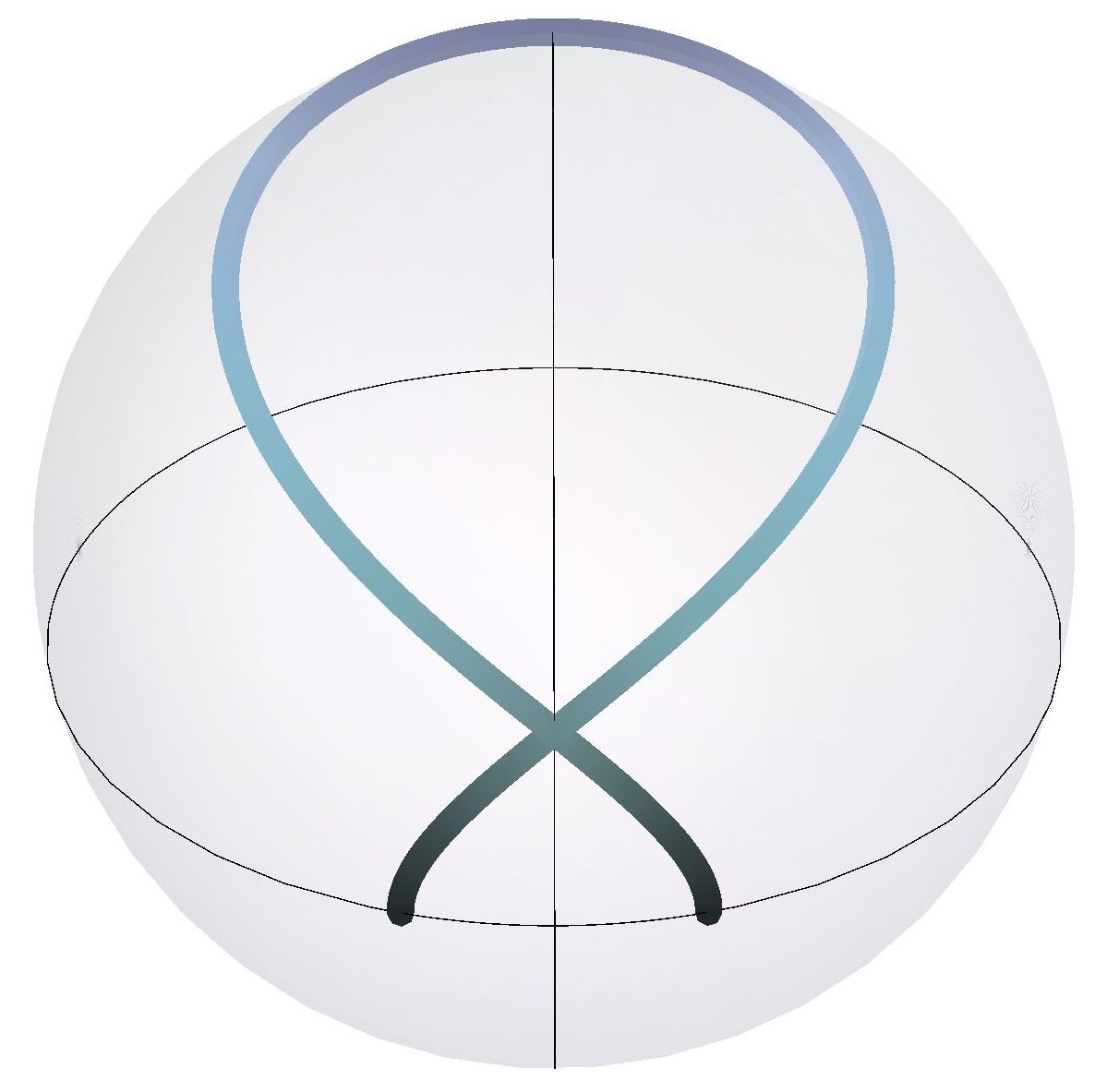}\quad\quad\includegraphics[width=.16\textwidth]{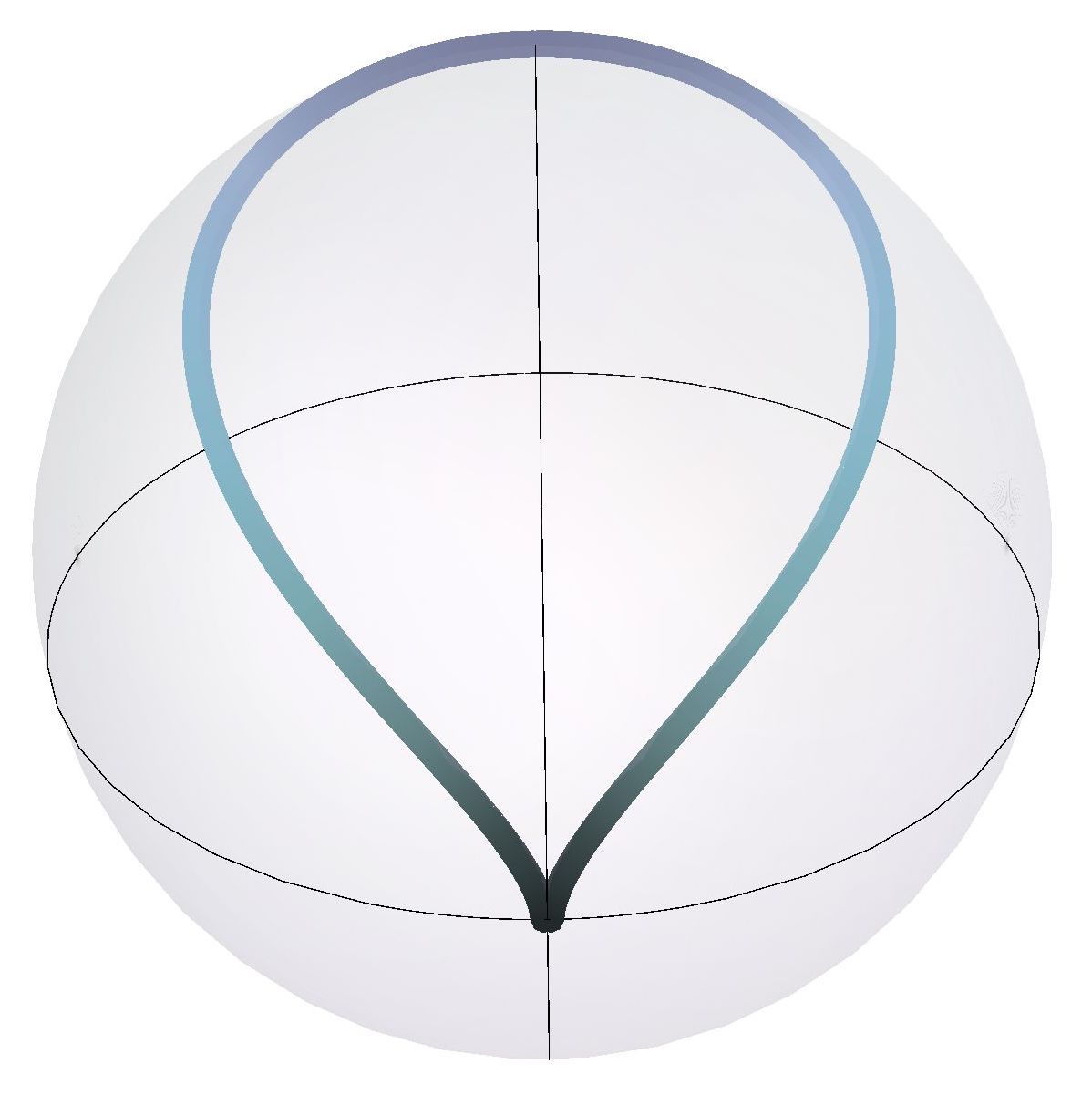}\quad\quad\includegraphics[width=.166\textwidth]{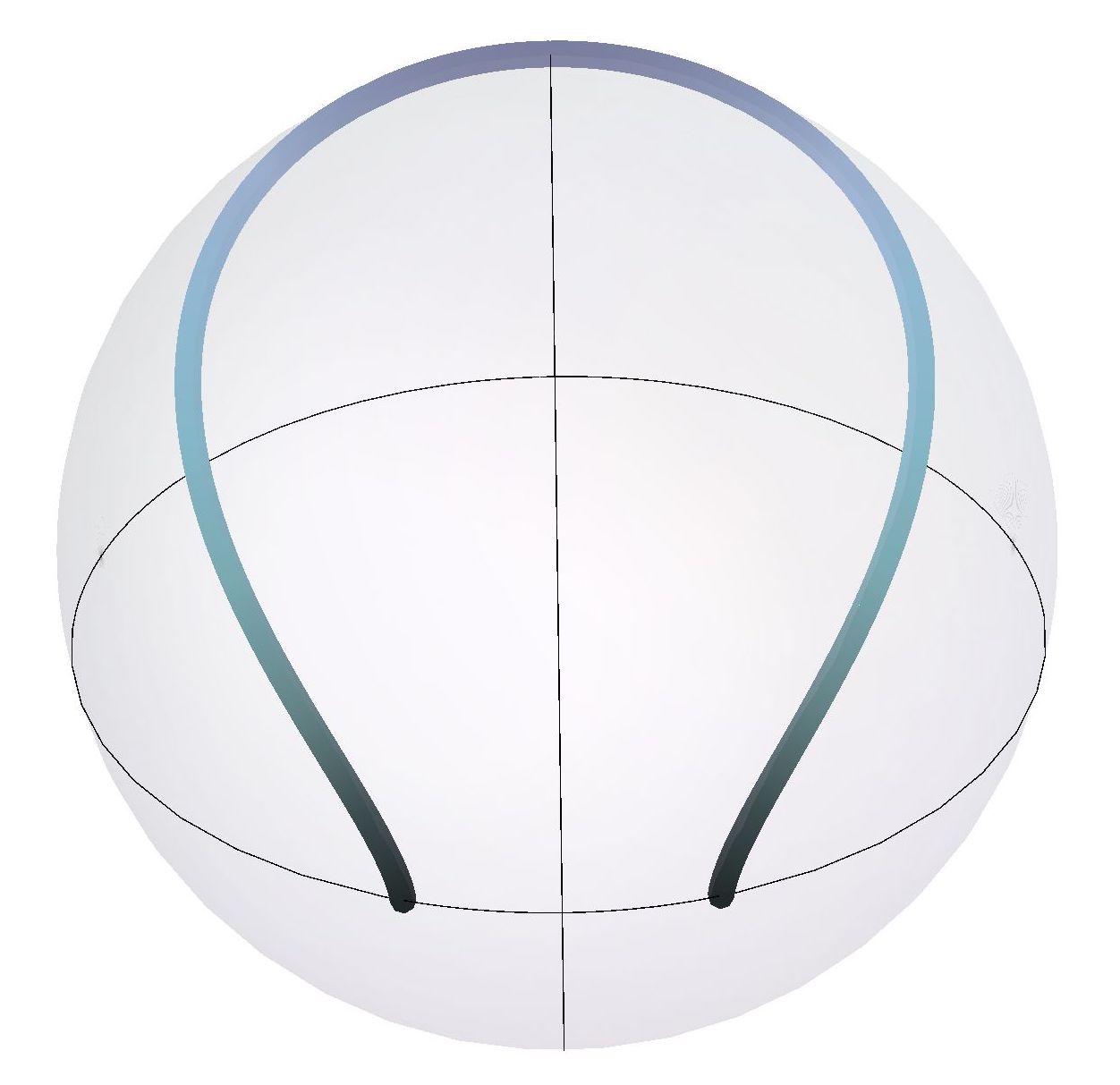}\quad\quad\includegraphics[width=.16\textwidth]{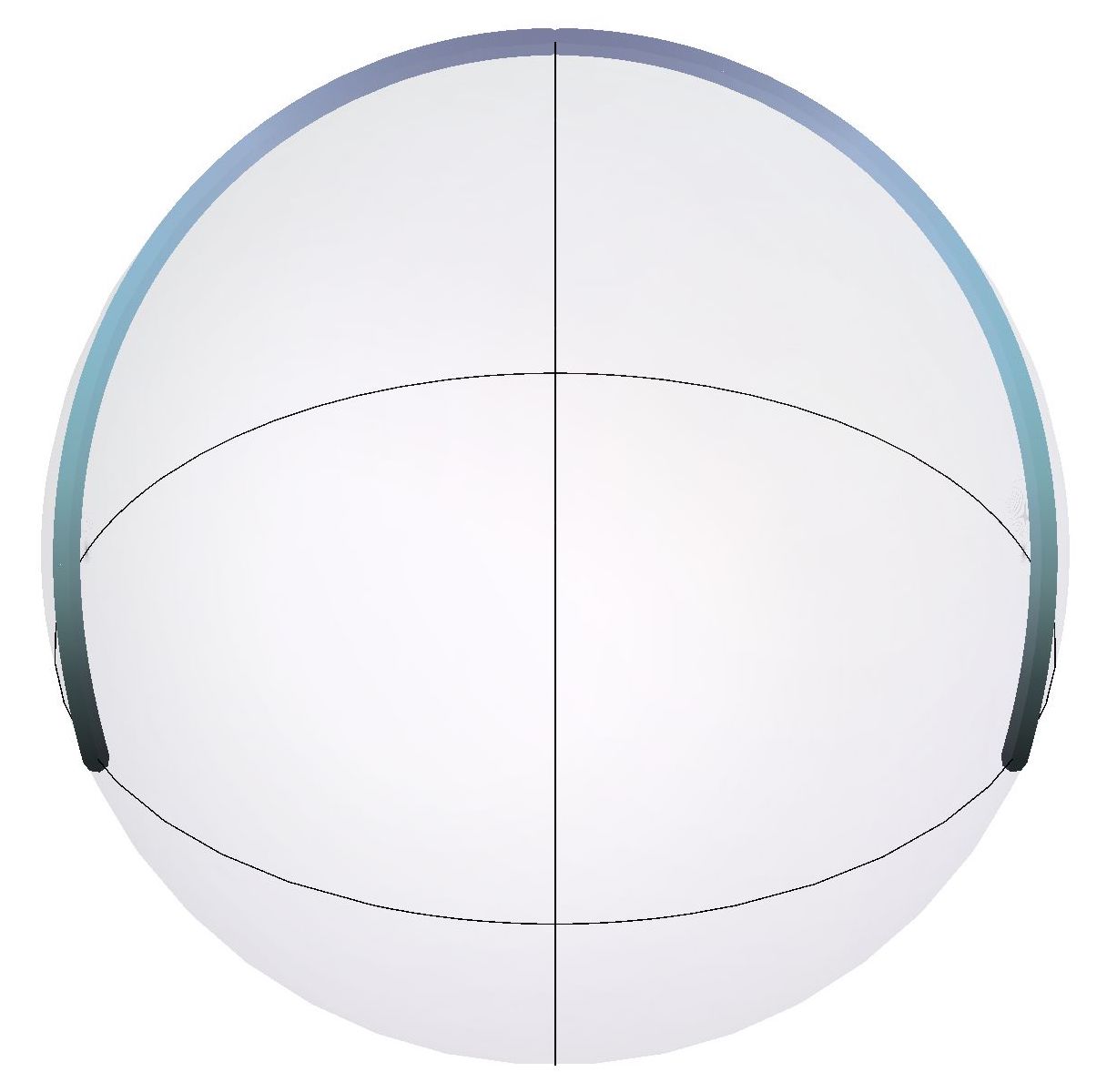}
\end{center}
\caption{Critical curves of $\mathbf{\Theta}_\mu$ with non-constant curvature in the sphere $\s^2(1)$ passing through the pole. From left to right: orbit-like type ($\mu\simeq 0.15$ and $\mu=0.35$), borderline type ($\mu\simeq 0.402$),  non-simple biconcave type ($\mu=0.42$ and $\mu=0.499$), figure-eight type ($\mu\simeq 0.54$),  simple biconcave type ($\mu=0.6$) and oval type ($\mu=1$). In all these cases $d=\mu^2e^2$.}
\label{criticalS2Pole}
\end{figure}

\subsection{Case $\rho<0$: the hyperbolic plane  $\h^2(\rho)$}

After rescaling and change of orientation, we can consider that $\rho=-1$ and $\mu>0$.   Recall that we are assuming  $d>0$. In this case, from \eqref{param}, a critical curve $\gamma$ parametrizes as 
\begin{equation}\label{paramh2}
\gamma(x)=\frac{1}{\sqrt{d}}\left(\mu x,\sqrt{d+\mu^2 x^2}\sinh\left(\sqrt{d}\,\psi(x)\right),\sqrt{d+\mu^2x^2}\cosh\left(\sqrt{d}\,\psi(x)\right)\right),
\end{equation}
where now we must use $\rho=-1$ in the expression \eqref{Psi(x)} of $\psi(x)$.

The classification of  the critical curves of $\mathbf{\Theta}_\mu$ in $\h^2(-1)$ is the following: see figure \ref{criticalH2}. 

\begin{theorem}\label{r5}
 Let $\mu>0$ and $d>0$ be arbitrary constants. The critical curves with non-constant curvature of $\mathbf{\Theta}_\mu$ in $\h^2(-1)$ for $d>0$ are of oval type, simple biconcave type, figure-eight type,  non-simple biconcave type, borderline type and orbit-like type.
\end{theorem}
\begin{proof} For any constant $\mu>0$, the values of $x_{\pm}$ satisfy $0<x_{-}<1<e<x_{+}$. Moreover, we   have $\widehat{F}(x_{+})<0<\widehat{F}(x_{-})$. Hence, the orbit $F(x,y)=d$ has exactly three cuts with the axis $y=0$ if $d<\widehat{F}(x_{-})$. Otherwise, we have only one cut (recall that we are not considering the critical point $P_{-}$). 

We argue as in Theorem \ref{r2}. By Proposition \ref{sym}, we are only considering the branch $y\geq 0$ in the associated orbit. There are four different cases depending on the value of $d>0$ and the type of orbit described by the value $x_0$ (see figure \ref{orbitas}, left):
\begin{enumerate}
\item Case $d\leq \widehat{F}(x_{-})$ and $x_0\leq x_{-}$ (see purple orbits). The curve starts at $x=0$ meeting orthogonally the geodesic $\beta$. While $x$ increases until $x_0$, $\psi(x)$, \eqref{Psi(x)}, which represents the hyperbolic variation angle of $\gamma(x)$, \eqref{paramh2}, decreases. Thus, it is of oval type.
\item Case $d=\widehat{F}(x_{-})$ and $x_0>x_{-}$ (see the green orbit). In this case, using $\widehat{F}(x)$, it is easy to prove that $x>x_{-}>0$ so that the curve never cuts (or tends to) the axis $\beta$. In fact, $\gamma$ is asymptotic in its end-points to the critical hypercycle represented by $P_{-}$. While $x$ increases until $x_0$, the function $\psi(x)$, \eqref{Psi(x)}, decreases from $x=x_{-}$ to $x=e$, then it starts to increase. Notice that the $x_2$ component vanishes precisely at $x=x_0$ and, as a consequence, the curve cuts the geodesic $\alpha$ in that point. Finally, after symmetry, $\gamma$ has a single loop. Thus, it is of borderline type.
\item Case $d>\widehat{F}(x_{-})$ (see blue orbits). These curves begin at $x=0$ meeting orthogonally the geodesic $\beta$. While $x$ increases until $x_0$ the behavior of the function $\psi(x)$, \eqref{Psi(x)}, is the same as in previous point. Then, depending on the sign of $\psi(0)$ we have three different types:
\begin{enumerate}
\item Case $\psi(0)>0$. Since the hyperbolic variation angle in \eqref{paramh2} decreases from a positive value of the $x_2$ component at $x=0$ and then increases until reaching $x_2=0$, there is an intermediate point where $\gamma$ cuts the geodesic $\alpha$. Hence, the curve is of non-simple biconcave type.
\item Case $\psi(0)=0$. Here, $\gamma$ cuts $\alpha$ only at $x=0$ and at $x=x_0$. Thus, it is of figure-eight type.
\item Case $\psi(0)<0$. The component $x_2$ in \eqref{paramh2} has a negative value at $x=0$ and the hyperbolic variation angle gets more negative until $x=e$. Then, it increases until reaching $x_2=0$ at $x=x_0$. Necessarily the curve is simple and, as a consequence, of  simple biconcave type.
\end{enumerate}
\item Case $0<d<\widehat{F}(x_{-})$ and $x_0>x_{+}$ (see brown orbits). The curvature $\kappa$ of $\gamma$ for these values is periodic since the corresponding orbits are closed. The parameter $x$ is bounded from below by $x_{-}>0$ which means that $\gamma$ does not meet the geodesic $\beta$. It is also bounded from above by $x_0$. Hence $\gamma$ lies between two hypercycles. Therefore, the curve is of orbit-like type. 
\end{enumerate}
This covers all the possible cases.
\end{proof}

\begin{figure}[hbtp]
\begin{center}
\includegraphics[width=.16\textwidth]{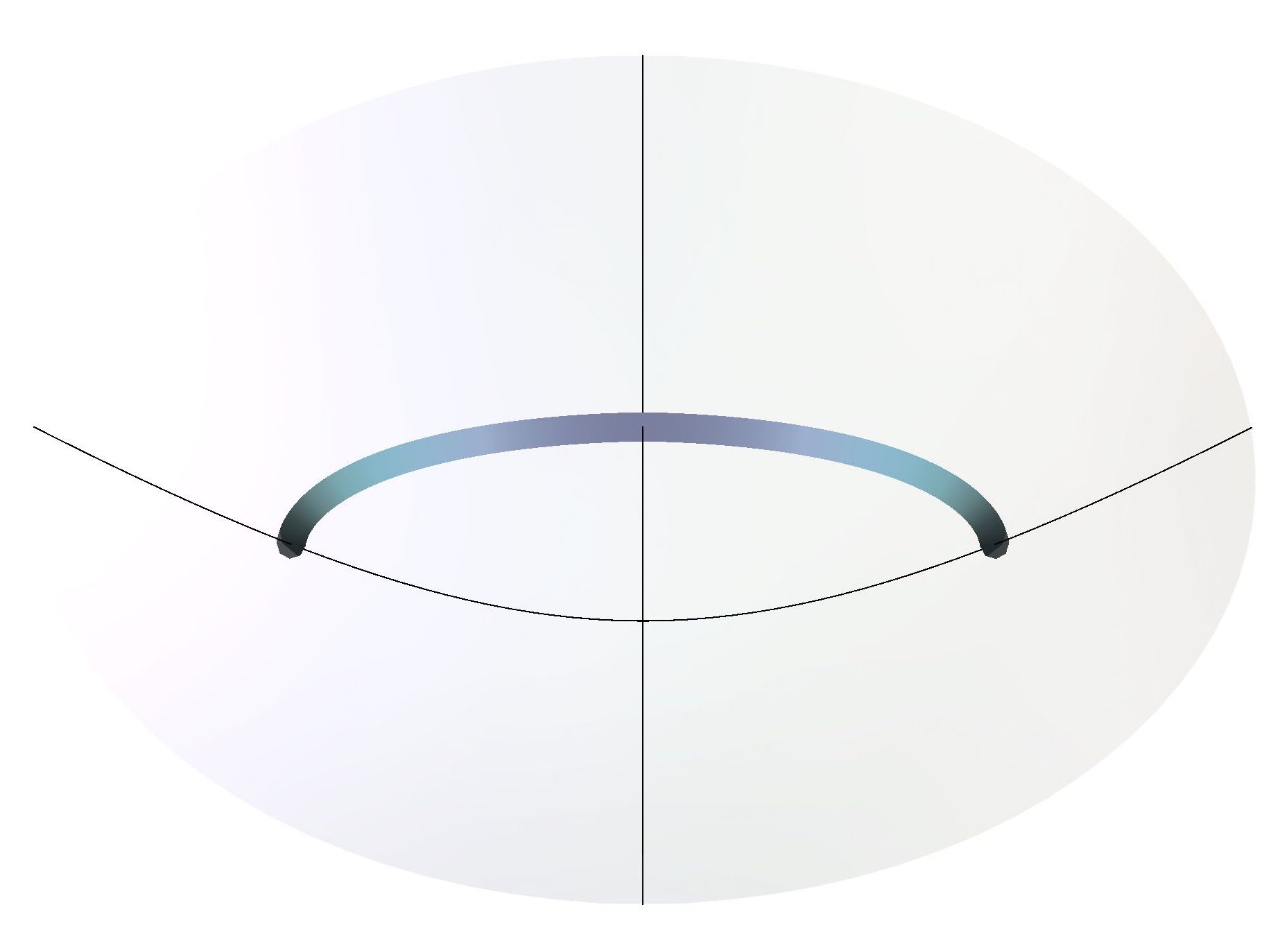}\,\includegraphics[width=.16\textwidth]{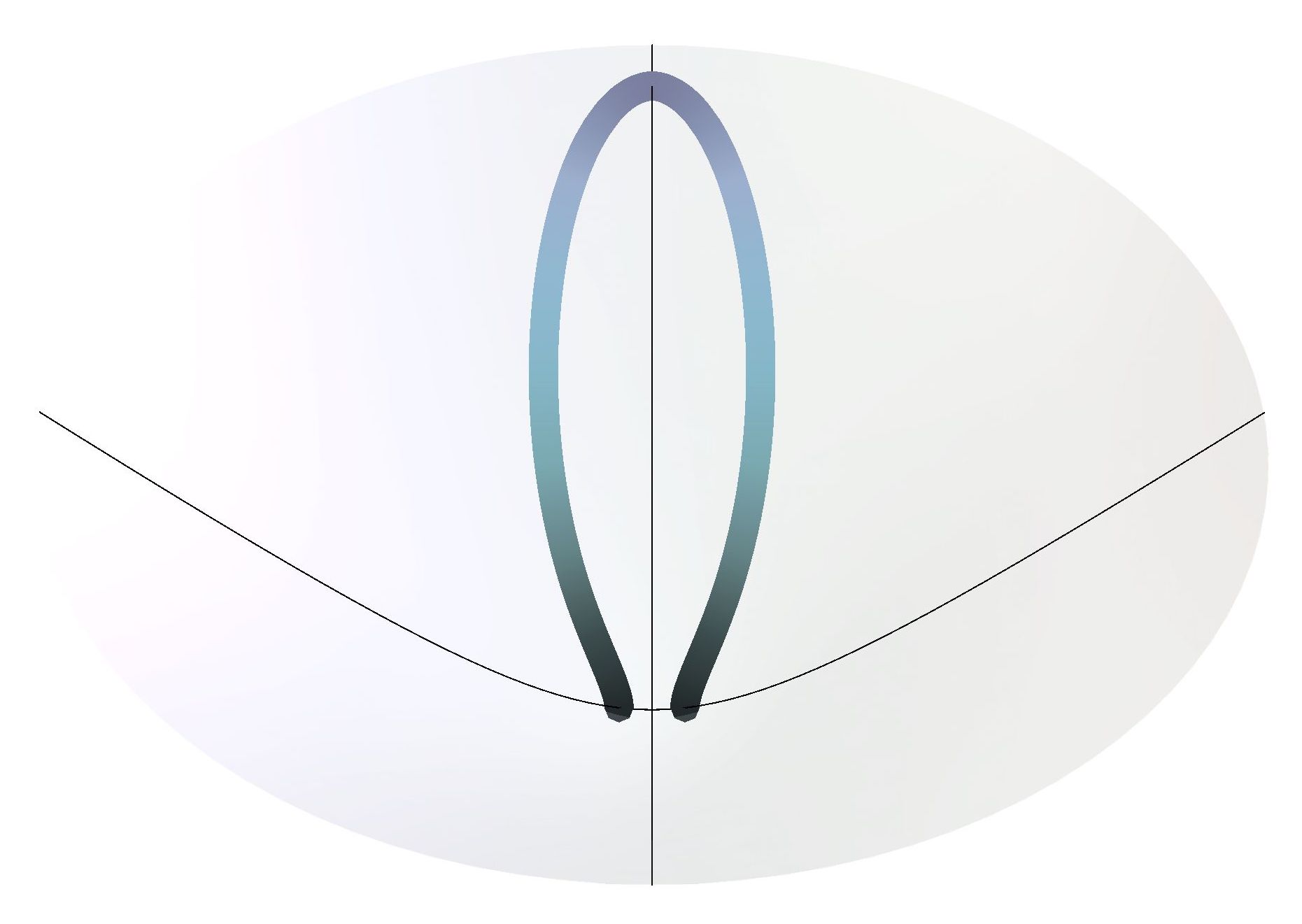}\,\includegraphics[width=.16\textwidth]{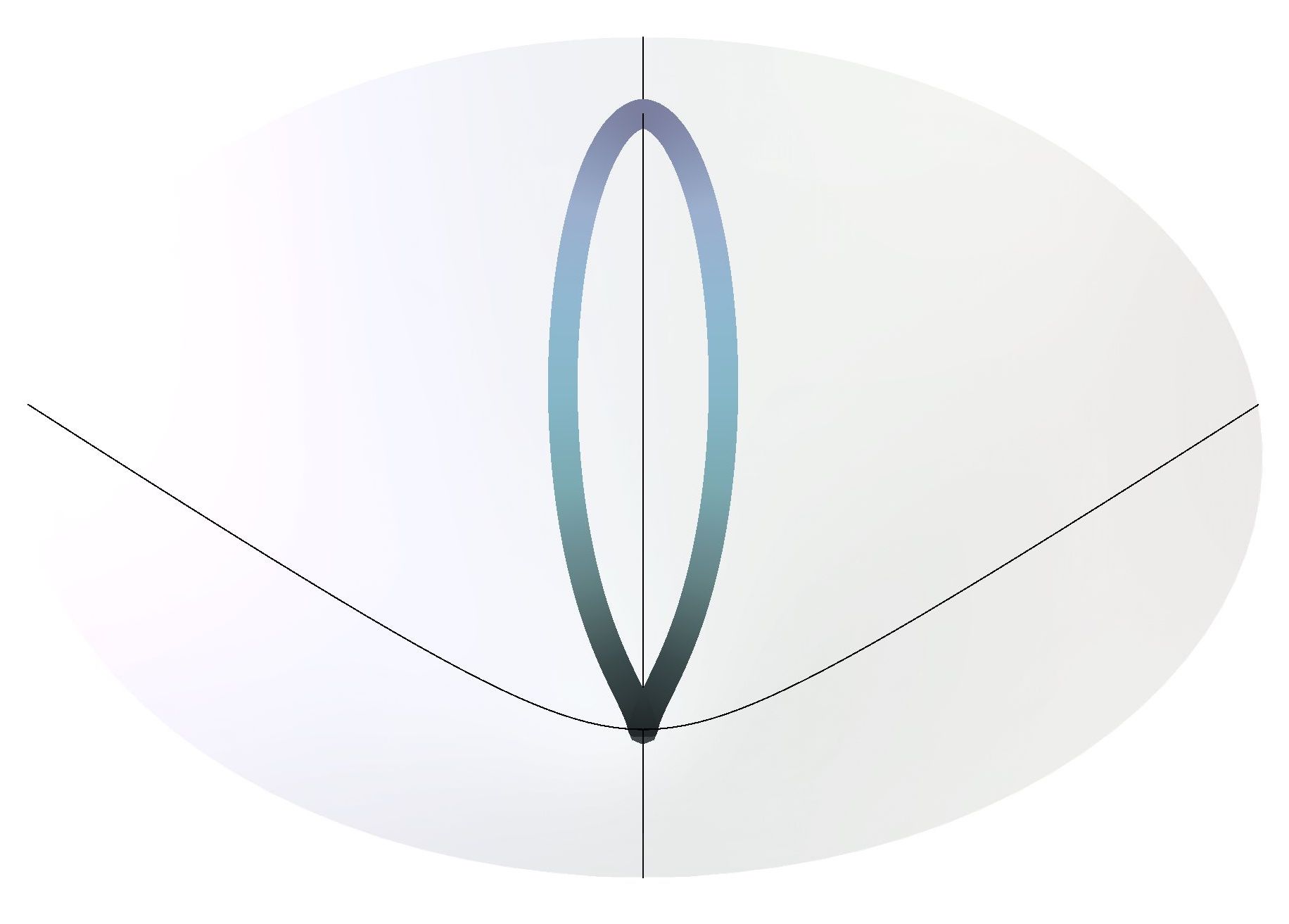}\,\includegraphics[width=.16\textwidth]{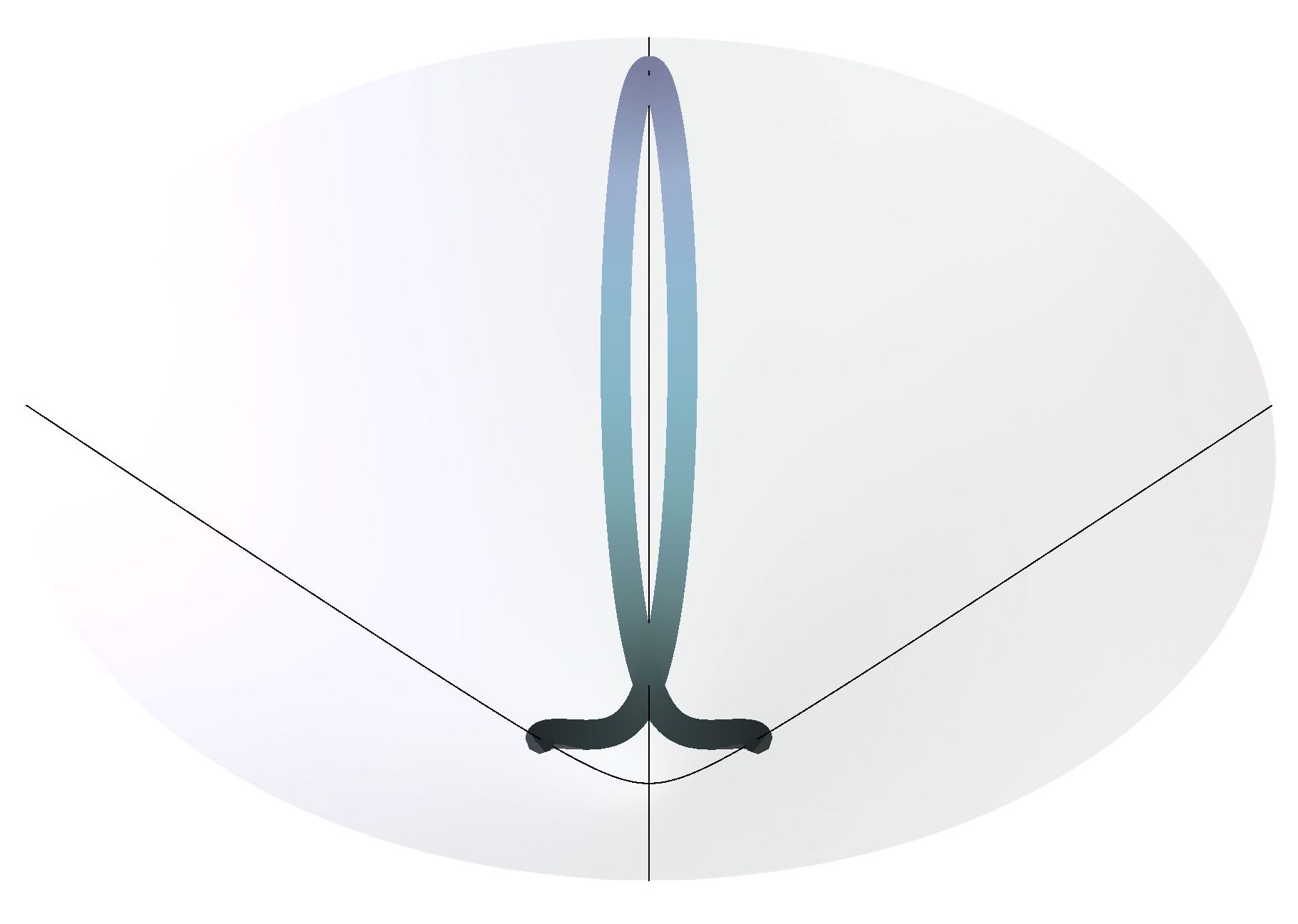}\,\includegraphics[width=.165\textwidth]{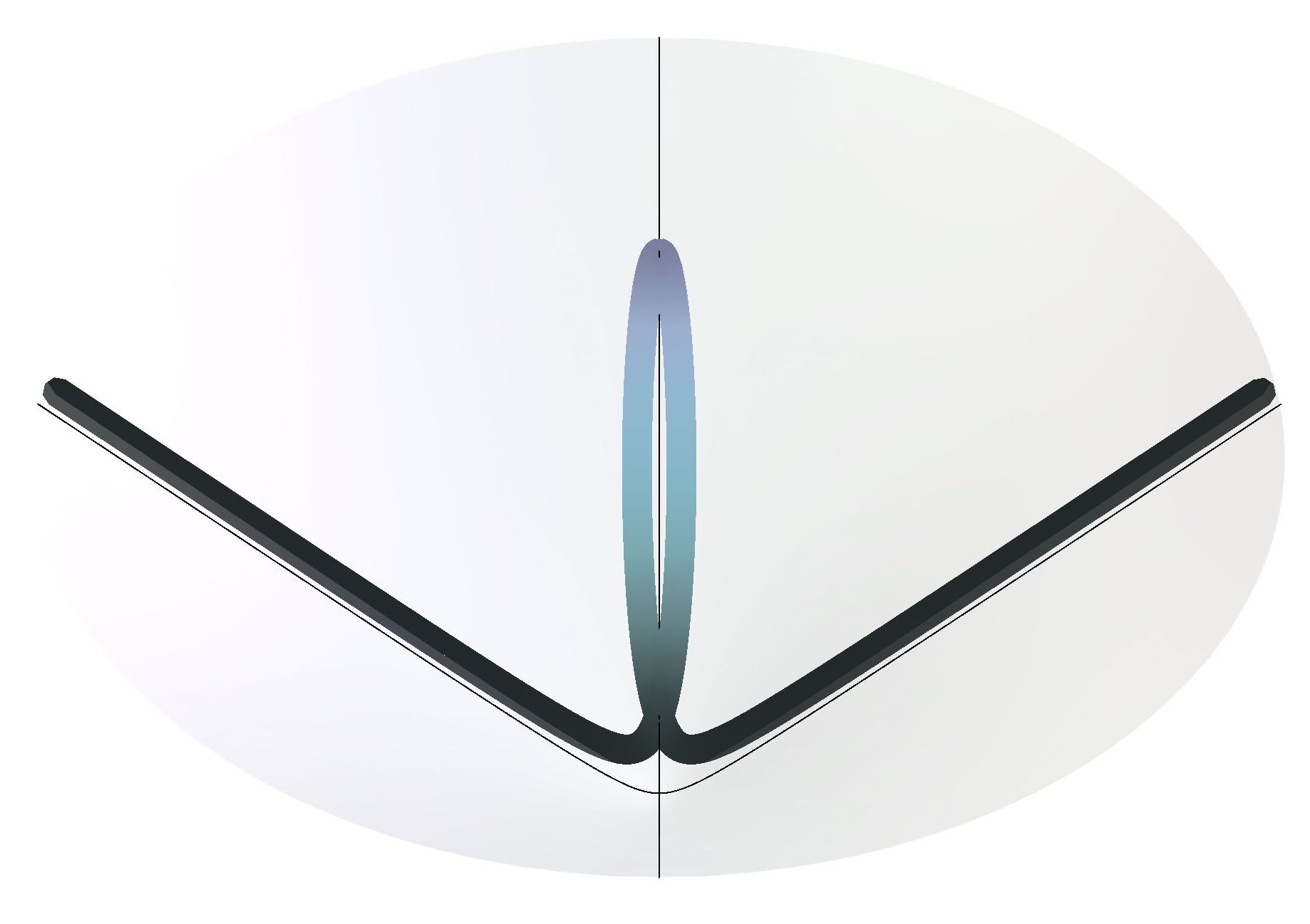}\,\includegraphics[width=.165\textwidth]{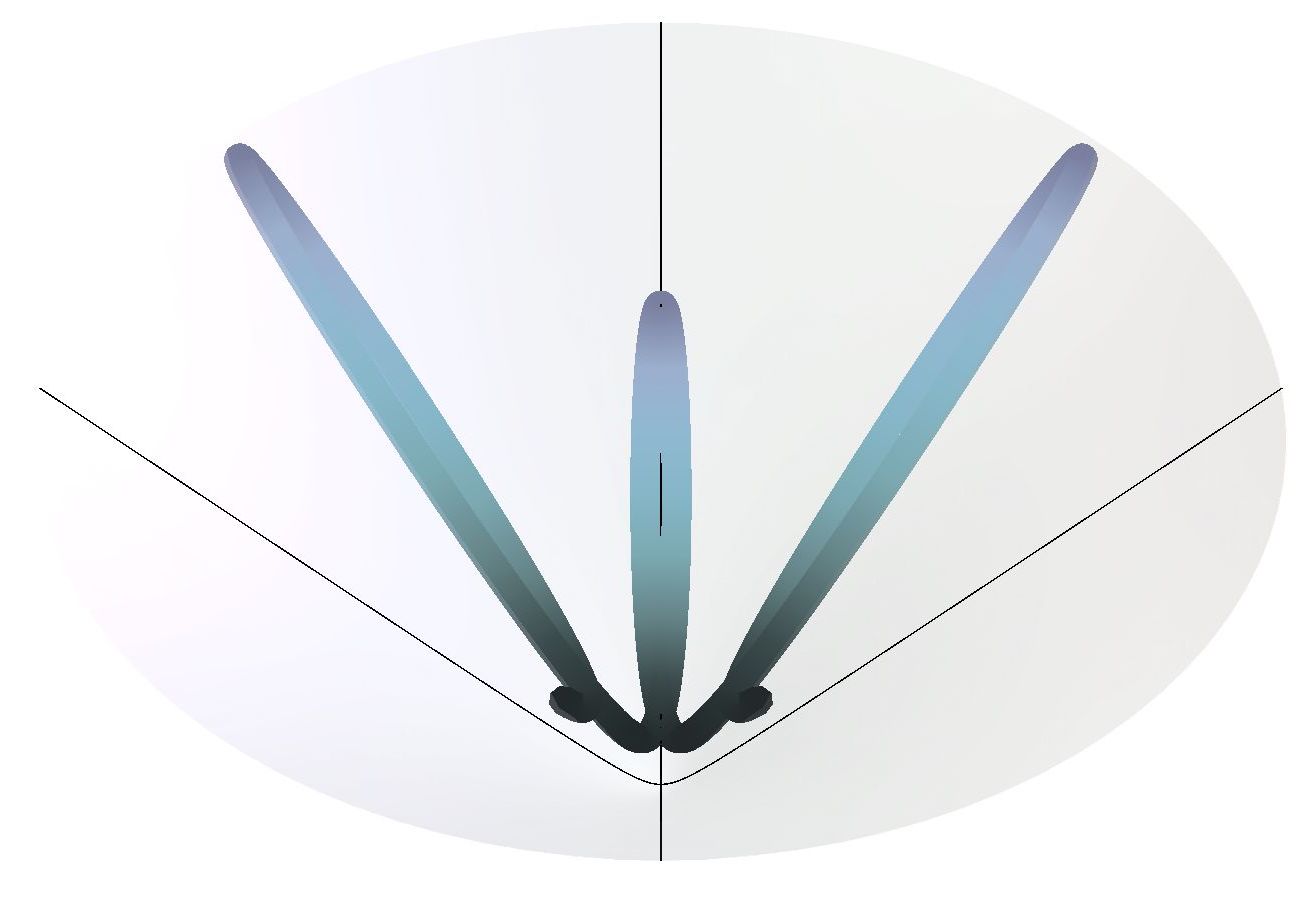}
\end{center}
\caption{Critical curves of $\mathbf{\Theta}_\mu$ with non-constant curvature in the sphere $\h^2(-1)$. The geodesic $\alpha$ is represented by the vertical hyperbola, whereas $\beta$ is the horizontal one. Here $\mu=1$ and, from left to right:  oval type ($d=0.3$),  simple biconcave type ($d=8$), figure-eight type ($d\simeq 4.5$),  non-simple biconcave type ($d=0.6$), borderline type ($d\simeq 0.47$) and orbit-like type ($d=0.3$).}
\label{criticalH2}
\end{figure}

\section*{Acknowledgements}
Rafael L\'opez has been partially supported by the grant no. MTM2017-89677-P, MINECO/ AEI/FEDER, UE.
\'Alvaro P\'ampano has been partially supported by MINECO-FEDER grant PGC2018-098409-B-100, Gobierno Vasco grant IT1094-16 and Programa Posdoctoral del Gobierno Vasco, 2018.

\end{document}